\documentclass[preprint,11pt]{imsart}
\pdfoutput=1
\RequirePackage[OT1]{fontenc}
\usepackage{amsthm,amsmath,amssymb,natbib,bm,enumitem}
\RequirePackage[colorlinks,citecolor=blue,urlcolor=blue]{hyperref}

\startlocaldefs
\numberwithin{equation}{section}
\theoremstyle{plain}
\newtheorem{thm}{Theorem}[section]
\newtheorem{lemma}{Lemma}[section]
\newtheorem{corollary}{Corollary}[section]
\theoremstyle{definition}

\theoremstyle{remark}
\newtheorem{remark}{Remark}[section]

\newtheorem*{assumptionpi}{Assumptions on $\pi$}
\newtheorem*{assumptionf}{Assumptions on $f$}
\endlocaldefs

\def\citeapos#1{\citeauthor{#1}'s (\citeyear{#1})}
\def\T{{ \mathrm{\scriptscriptstyle T} }}
\newcommand{\rd}{\mathrm{d}}
\newcommand{\oldiff}{\overline{\mathrm{diff}B}}
\newcommand{\ndiff}{\mathrm{diff}B}
\DeclareMathOperator*{\argmin}{arg\,min}

\newcommand{\dps}{\displaystyle}
\newcommand{\AKA}[1]{\textcolor{black}{#1}}
\allowdisplaybreaks
\begin{document}

\begin{frontmatter}
 \title{Admissible Bayes equivariant estimation of location vectors for spherically symmetric
distributions with unknown scale}
\runtitle{Admissible Bayes equivariant estimation}

\begin{aug}
\author{\fnms{Yuzo} \snm{Maruyama}\thanksref{t1,m1}
\ead[label=e1]
{maruyama@csis.u-tokyo.ac.jp}}
 \and
\author{\fnms{William, E.} \snm{Strawderman}
\thanksref{t2,m2}
\ead[label=e2]{straw@stat.rutgers.edu}}

\thankstext{t1}{This work was partially supported by KAKENHI \#25330035, \#16K00040.}
\thankstext{t2}{This work was partially supported by grants from the Simons Foundation (\#209035 and \#418098 to William Strawderman).}
\address{University of Tokyo\thanksmark{m1} and Rutgers University\thanksmark{m2} \\
\printead{e1,e2}}
\runauthor{Y. Maruyama and W. Strawderman}

\end{aug}

\begin{abstract}
 This paper investigates estimation of the mean vector under invariant quadratic loss 
 for a spherically symmetric location family with a residual vector with density
 of the form $ f(x,u)=\eta^{(p+n)/2}f(\eta\{\|x-\theta\|^2+\|u\|^2\})$\AKA{, where $\eta$ is unknown.
We show that the natural estimator $x$ is admissible for $p=1,2$. Also, for $p\geq 3$,}
we find classes of generalized Bayes estimators that are admissible
 within the class of equivariant estimators of the form
 $\{1-\xi(x/\|u\|)\}x$.
 In the Gaussian case,
 a variant of the James--Stein estimator, $[1-\{(p-2)/(n+2)\}/\{\|x\|^2/\|u\|^2+(p-2)/(n+2)+1\}]x$,
which dominates the natural estimator $x$, is also admissible within this class.
 We also study the related regression model.
 \end{abstract}


\end{frontmatter}
\section{Introduction}
\label{sec:intro}
Let
\begin{equation}\label{eq:density.f}
 (X,U)\sim \eta^{(p+n)/2}f(\eta\{\|x-\theta\|^2+\|u\|^2\})
\end{equation}
where $X\in\mathbb{R}^p$ and $U\in\mathbb{R}^n$ and where $\theta\in\mathbb{R}^p$ and $\eta\in\mathbb{R}_+$ are unknown.
We \AKA{mainly} assume
\begin{equation}\label{assumption.p.n}
\AKA{p\geq 3 \text{ and }n\geq 2},
\end{equation}
and consider the problem of estimating $\theta$ under scaled quadratic loss
\begin{equation}\label{eq:loss}
 L(\delta,\theta,\eta)=\eta\|\delta - \theta\|^2.
\end{equation}
In particular, we are interested in the admissibility
among the class of equivariant estimators of the form
\begin{equation}\label{eq:equiv.est.2.intro}
 \delta_{\xi}(X,U)=\left\{1-\xi(X/\|U\|)\right\}X, \text{ where }\xi:\mathbb{R}^p\to\mathbb{R}.
\end{equation}
We assume that $f(\cdot)\geq 0$ 
is defined so that each coordinate has variance $1/\eta$.
In particular, this implies that $f(\cdot)$ in \eqref{eq:density.f}, 
satisfies
\begin{equation}\label{assmp.f.1}
 \int_{\mathbb{R}^{p+n}}f(\|v\|^2) \rd v=1, \ \int_{\mathbb{R}^{p+n}}v_i^2f(\|v\|^2) \rd v=1, 
\end{equation}
for $v=(v_1,\dots,v_{p+n})^\T\in\mathbb{R}^{p+n}$.
Needless to say, this is a generalization of the Gaussian case where
\begin{equation*}
 f_G(t)=\frac{1}{(2\pi)^{(p+n)/2}}\exp(-t/2)
\end{equation*}
and hence $X\sim N_p(\theta,\eta^{-1}I)$ and $\|U\|^2\sim\eta^{-1}\chi^2_n$ are mutually independent.

In Section \ref{sec:group},
we show that if an estimator of the form
\begin{equation}\label{eq:equiv.est.1.intro}
 \delta_{\psi}(X,U)=\left\{1-\psi(\|X\|^2/\|U\|^2)\right\}X,
 \text{ where }\psi:\mathbb{R}_+\to\mathbb{R},
\end{equation}
is admissible within the class of all such estimators, then it is also
admissible within a larger class of estimators, the class of all estimators of the form
$\delta_{\xi}(X,U)$ given by \eqref{eq:equiv.est.1.intro}.
Note that the risk of an equivariant estimator of the form \eqref{eq:equiv.est.1.intro}
is a function of $\lambda=\eta\|\theta\|^2$.
Section \ref{sec:equiv} studies equivariant estimators of the form \eqref{eq:equiv.est.1.intro}
which minimize the average risk with respect to a (proper) prior $\pi(\lambda)$ on
the maximal invariant $\lambda=\eta\|\theta\|^2$.
We  give an expression for this average risk and for the equivariant estimator
which effects the minimization. Additionally we show that this proper Bayes equivariant
estimator is equivalent to the generalized (and not proper) Bayes estimator
corresponding to the prior on $(\theta,\eta)$
\begin{equation*}
 \pi(\theta,\eta)=\eta^{-1}\{\eta\|\theta\|^2\}^{1-p/2}\pi(\eta\|\theta\|^2).
\end{equation*}
Further we demonstrate that such an estimator is admissible among the class of
estimators of the form \eqref{eq:equiv.est.1.intro} and hence \eqref{eq:equiv.est.2.intro}.

Section \ref{sec:equiv.blyth}, using \citeapos{Blyth-1951} method,
extends the class of estimators which are admissible within the class of 
estimators of the form \eqref{eq:equiv.est.1.intro}.
One main result gives admissibility under $\pi(\lambda)$ including
\begin{equation*}
 \pi(\lambda)=\lambda^\alpha \text{ for }-1/2<\alpha\leq 0,
\end{equation*}
for densities $f$ including the normal distribution
and many generalized multivariate $t$ distributions.
An interesting special case gives admissibility (within the class of equivariant estimators)
of the generalized Bayes equivariant estimator corresponding to $\pi(\lambda)\equiv 1$
or $\pi(\theta,\eta)=\eta^{-1}\{\eta\|\theta\|^2\}^{1-p/2}$.
Here the form of the generalized Bayes estimator is independent of the underlying
density $f(\cdot)$, as shown in \cite{Maruyama-2003b}. Further this estimator is minimax and dominates
the \cite{James-Stein-1961} estimator
\begin{equation*}
 \left(1-\frac{(p-2)/(n+2)}{\|X\|^2/\|U\|^2}\right)X
\end{equation*}
provided $f(\cdot)$ is non-increasing.
Another interesting result is on a variant of the James--Stein estimator
of the simple form
\begin{equation*}
 \left(1-\frac{(p-2)/(n+2)}{\|X\|^2/\|U\|^2+(p-2)/(n+2)+1}\right)X.
\end{equation*}
In the Gaussian case, this is the generalized Bayes equivariant estimator corresponding to 
\begin{equation*}
\pi(\lambda)=\lambda^{p/2-1}\int_0^\infty\frac{1}{(2\pi\xi)^{p/2}}\exp\left(-\frac{\lambda}{2\xi}\right)
\left(\frac{\xi}{1+\xi}\right)^{n/2}\rd \xi.
\end{equation*}
It is admissible within the class of equivariant estimators, 
and is minimax. 

In Section \ref{sec.regression}, we demonstrate that
our setting \eqref{eq:density.f} is regarded as a canonical form of a regression model
with an intercept and a general spherically symmetric error distribution,
where estimators of the form \eqref{eq:equiv.est.1.intro}
corresponds to 
estimators of the vector of regression coefficients of the form
$\{1-\psi_\star(R^2)\}\hat{\beta}$ where $\psi_\star:(0,1)\to\mathbb{R}$,
$\hat{\beta}$ is the vector of
least square estimators, and $R^2$ is the coefficient of determination.
Also estimators of the form \eqref{eq:equiv.est.2.intro} corresponds to
estimators of the vector of regression coefficients of the form
$\{1-\xi_\star(\mathsf{t})\}\hat{\beta}$ where $\xi_\star:\mathbb{R}^p\to\mathbb{R}$ and
$\mathsf{t}$ is the vector of $t$ values.
Hence, from the regression viewpoint, these two classes of equivariant estimators
are quite natural.

Section \ref{sec.cr} gives some concluding remarks.
Most of the proofs are given in a series of Appendices.

\AKA{In this paper, we consider admissibility within the (restricted) class of
equivariant estimators.
The ultimate goal in this direction is clearly to show that some equivariant 
estimators of the form \eqref{eq:equiv.est.1.intro} are admissible among
all estimators. Actually, when $p=1,2$, the natural estimator $X$ is admissible
among all estimators, as shown in Appendix \ref{p12}.
(This is not surprising but very expected. 
As far as we know, however, admissibility of $X$ with nuisance unknown scale $\eta$, has not yet
reported in any major journals and hence we provide the proof.)
When the Stein phenomenon with $p\geq 3$ occurs,
considering general admissibility of equivariant estimators
in the presence of nuisance parameter(s) has been a longstanding unsolved
problem as mentioned in \cite{James-Stein-1961} and \cite{Brewster-Zidek-1974}.
While our results with $p\geq 3$ do not resolve the general admissibility issue,
they do advance substantially our understanding of admissibility within the class
of equivariant estimators.}

\section{Admissibility in a broader sense}
\label{sec:group}
We consider two groups of transformations. In the following, let $S=\|U\|^2$.
\begin{description}
\item[Group I]\label{gI} 
 \begin{align*}
  X\to \gamma\Gamma X, \quad \theta\to \gamma\Gamma \theta, \quad S\to \gamma^2S, \quad \eta\to \eta/\gamma^2,
 \end{align*}
      where $\Gamma\in  \mathcal{O}(p)$, the group of $p\times p$ orthogonal matrices,
      and $\gamma\in\mathbb{R}_+$.
 \item[Group II]\label{gII}
 \begin{align*}
  X\to \gamma X, \quad \theta\to \gamma \theta, \quad S\to \gamma^2S, \quad \eta\to \eta/\gamma^2,
 \end{align*}
where $\gamma\in\mathbb{R}_+$.
\end{description}
Equivariant estimators for Group I should satisfy
\begin{align*}
  \delta(\gamma\Gamma X,\gamma^2S)=\gamma\Gamma \delta(X,S),
\end{align*}
and reduce to estimators of the form 
\begin{equation}\label{eq:equiv.est.0}
 \delta_{\psi}=\left\{1-\psi(\|X\|^2/S)\right\}X
\end{equation}
where $\psi:\mathbb{R}_+\to\mathbb{R}$.
Equivariant estimator for Group II should satisfy
\begin{align*}
  \delta(\gamma X,\gamma^2S)=\gamma \delta(X,S),
\end{align*}
and reduce to estimator of the form 
\begin{equation}\label{eq:equiv.est.1}
 \delta_{\xi}=\left\{1-\xi(X/\sqrt{S})\right\}X
\end{equation}
where $\xi:\mathbb{R}^p\to\mathbb{R}$.
It is useful to note the following.
\begin{lemma}\label{lem:risk.form}
 \begin{enumerate}
  \item 
The risk, $R(\theta,\eta, \delta_{\psi})=E\left[\eta\|\delta_{\psi}-\theta\|^2\right]$,
of an estimator $\delta_{\psi}$, is a function of $\eta\|\theta\|^2\in\mathbb{R}_+$.
\item 
The risk, $R(\theta,\eta, \delta_{\xi})$,
of an estimator $\delta_{\xi}$, is a function of $\eta^{1/2}\theta\in\mathbb{R}^p$.
 \end{enumerate}
\end{lemma}
The standard proof is left to the reader.

Let two classes of estimators be
\begin{align*}
 \mathcal{D}_\psi&=\left\{\delta_{\psi} \text{ with }\psi: \mathbb{R}_+\to\mathbb{R} \text{ given by \eqref{eq:equiv.est.0}}\right\}, \\
 \mathcal{D}_\xi&=\left\{\delta_{\xi}\text{ with }\xi: \mathbb{R}^p\to\mathbb{R}  \text{ given by \eqref{eq:equiv.est.1}}\right\}.
\end{align*}
Clearly it follows that $ \mathcal{D}_\psi\subset \mathcal{D}_\xi$.
We shall show that if $\delta \in \mathcal{D}_\psi$
is admissible among the class $\mathcal{D}_\psi$, 
then it is admissible among the class $\mathcal{D}_\xi$. 
The proof is due to Section 3 of \cite{Stein-1956},
based on the compactness of the orthogonal group $\mathcal{O}(p)$,
and the continuity of, the problem.

\begin{thm}\label{thm:stein.1956}
If $\delta \in \mathcal{D}_\psi$ is admissible among the class $\mathcal{D}_\psi$, 
then it is admissible among the class $\mathcal{D}_\xi$.
\end{thm}
\begin{proof}
See Appendix \ref{sec:ap.stein.1956}.
\end{proof}
 In this paper, we will investigate admissibility among the class $\mathcal{D}_\psi$.
 Admissibility admissibility among the class $\mathcal{D}_\xi$ then follows by Theorem
 \ref{thm:stein.1956}.

 \section{Proper Bayes equivariant estimators}
 \label{sec:equiv}
Recall that an equivariant estimator for Group I is given by
\begin{equation}\label{eq:equiv.est}
 \delta_{\psi}=\left\{1-\psi(\|X\|^2/S)\right\}X.
\end{equation}
Since, as noted in Lemma \ref{lem:risk.form},
the risk function of the estimator $\delta_\psi \in \mathcal{D}_\psi$,
$ R(\theta,\eta,\delta_{\psi})$, depends only on 
$\eta\|\theta\|^2\in\mathbb{R}_+$, 
it may be expressed as
\begin{equation}\label{eq:thetaeta}
 R(\theta,\eta,\delta_{\psi})=\tilde{R}(\eta\|\theta\|^2, \delta_{\psi}).
\end{equation}
Let $\lambda=\eta\|\theta\|^2\in\mathbb{R}_+$.
We assume the prior density on $\lambda$ is $\pi(\lambda)$, and
in this section, we assume the propriety of $\pi(\lambda)$, that is,
\begin{equation}\label{eq:propriety}
 \int_0^\infty\pi(\lambda) \rd \lambda=1.
\end{equation}
For an equivariant estimator $\delta_\psi$, we define
the Bayes equivariant risk as
\begin{equation}\label{eq:eBrisk}
B(\delta_\psi, \pi)=\int_0^\infty \tilde{R}(\lambda,\delta_{\psi})\pi(\lambda) \rd \lambda.
\end{equation}
In this paper, the estimator $\delta_{\psi}$ which minimizes $B(\delta_\psi, \pi)$,
is called the Bayes equivariant estimator and is denoted by $\delta_\pi$.
In the following, 
let
\begin{equation}\label{eq:c_n}
 c_m=\pi^{m/2}/\Gamma(m/2)\text{ for }m\in\mathbb{N}_+
\end{equation}
and
\begin{equation}\label{pipi*}
 \bar{\pi}(\lambda)=c_p^{-1}\lambda^{1-p/2}\pi(\lambda)
\end{equation}
so that $\bar{\pi}(\|\mu\|^2)$ is a proper probability density on $\mathbb{R}^p$, that is,
\begin{equation}\label{eq:propriety.0}
 \int_{\mathbb{R}^p}\bar{\pi}(\|\mu\|^2)\rd \mu=1.
\end{equation}
\begin{thm}\label{thm:proper.admissible}
Assume $\int_0^\infty\pi(\lambda) \rd \lambda=1$ and that $f$ satisfies \eqref{assmp.f.1}.
 \begin{enumerate}
  \item \label{thm:proper.admissible.0}
	The Bayes equivariant risk $B(\delta_\psi, \pi)$, \eqref{eq:eBrisk}, is given by
\begin{equation}\label{eq:Bayesequivrisk.thm}
 \begin{split}
 B(\delta_\psi, \pi) &=
  c_n\int_{\mathbb{R}^p}
  \psi(\|z\|^2)\left\{\psi(\|z\|^2) -2\left(1-\frac{z^\T M_2(z,\pi)}{\|z\|^2 M_1(z,\pi)}\right)\right\}\\
  &\qquad \times
  \|z\|^2 M_1(z,\pi) \rd z +p,
\end{split}
\end{equation}
      where $c_n$ is given by \eqref{eq:c_n} and 
\begin{equation}\label{eq:thm:proper.admissible.M1.M2}
	\begin{split}
  M_1(z,\pi)&=
 \iint \eta^{(2p+n)/2}f(\eta\{\|z-\theta\|^2+1\})
\bar{\pi}(\eta\|\theta\|^2) \rd \theta   \rd \eta,  \\
 M_2(z,\pi)&=
\iint \theta
  \eta^{(2p+n)/2} f(\eta\{\|z-\theta\|^2+1\}) \bar{\pi}(\eta\|\theta\|^2) \rd \theta  \rd \eta.
	\end{split}      
\end{equation}
  \item \label{thm:proper.admissible.1}
	Given $\pi(\lambda)$, the minimizer of $B(\delta_\psi, \pi)$ with respect to $\psi$ 
      is
\begin{equation}\label{eq:eq_Bayes_sol.thm}
\psi_\pi(\|z\|^2)= \argmin_\psi B(\delta_\psi, \pi) =1-\frac{z^\T M_2(z,\pi)}{\|z\|^2 M_1(z,\pi)}.
\end{equation}
  \item \label{thm:proper.admissible.2}
	The Bayes equivariant estimator
\begin{align}\label{eq:delta_pi_thm}
\delta_\pi=\left\{1- \psi_\pi(\|X\|^2/S)\right\}X
\end{align}
	is equivalent to the generalized Bayes estimator of $\theta$ with respect to
	the joint prior density $\eta^{-1}\eta^{p/2}\bar{\pi}(\eta\|\theta\|^2)$
	where $\bar{\pi}(\lambda)=c_p^{-1}\lambda^{1-p/2}\pi(\lambda)$.
  \item \label{thm:proper.admissible.3}
	The Bayes equivariant estimator $\delta_\pi$ is admissible among the class $\mathcal{D}_\psi$.
 \end{enumerate}
\end{thm}
\begin{proof}
 See Appendix \ref{sec:ap.proper.admissible}.
\end{proof}

\begin{remark}
 As shown in Appendix \ref{sec:ap.nu.neq.-1},
 the generalized Bayes estimator of $\theta$ with respect to
	the joint prior density $\eta^\nu\eta^{p/2}\bar{\pi}(\eta\|\theta\|^2)$
for any $\nu$ is a member of the class $\mathcal{D}_\psi$.
Part \ref{thm:proper.admissible.2} of Theorem \ref{thm:proper.admissible}
applies only to the special case of $\nu=-1$. 
The admissibility results of this section and of Section \ref{sec:equiv.blyth}
 apply only to this special case of $\nu=-1$ and imply neither admissibility
 or inadmissibility of generalized Bayes estimators if $\nu\neq -1$.
 Also note that while $\pi(\lambda)$ is assumed proper in this section,
 the prior on $(\theta,\eta)$, $\eta^{-1}\eta^{p/2}\bar{\pi}(\eta\|\theta\|^2)$, is never proper
 since
 \begin{align*}
  \int_0^\infty \!\!\int_{\mathbb{R}^p} \eta^{-1}\eta^{p/2}\bar{\pi}(\eta\|\theta\|^2)\rd\theta\rd \eta 
=\int_0^\infty \!\!\int_{\mathbb{R}^p} \eta^{-1}\bar{\pi}(\|\mu\|^2)\rd\mu\rd \eta 
=1\times\int_0^\infty \frac{\rd \eta}{\eta}=\infty.
 \end{align*}
\end{remark}

 \section{Admissible Bayes equivariant estimators through the Blyth method}
 \label{sec:equiv.blyth}
Even if $\pi(\lambda)$ on $\mathbb{R}_+$ (and hence $\bar{\pi}(\|\mu\|^2)$ on $\mathbb{R}^p$) is improper,
that is
\begin{equation}\label{eq:improper.pi.0}
\int_{\mathbb{R}^p}\bar{\pi}(\|\mu\|^2) \rd \mu
  =\int_0^\infty\pi(\lambda) \rd \lambda=\infty,
\end{equation}
the estimator $\delta_\pi$ given by \eqref{eq:delta_pi_thm} can be defined
if $M_1(z,\pi)$ and $M_2(z,\pi)$ given by \eqref{eq:thm:proper.admissible.M1.M2} 
are both finite, and the admissibility of $\delta_\pi$ within the class of
equivariant estimators can be investigated through the Blyth method.

We consider the Bayes equivariant risk difference under $\pi_i(\lambda)$
which is proper, but not necessarily standardized; i.e., $\int_0^\infty \pi_i(\lambda) \rd \lambda <\infty$.
Let $\delta_\pi$ and $\delta_{\pi i}$ be Bayes equivariant estimators with respect to
$\pi(\lambda)$ and $\pi_i(\lambda)$, respectively.
By Parts \ref{thm:proper.admissible.0} and \ref{thm:proper.admissible.1} of Theorem \ref{thm:proper.admissible}, 
the Bayes equivariant risk difference under $\pi_i(\lambda)$ is given as follows:
\begin{align}
& \ndiff(\delta_\pi,\delta_{\pi i};\pi_i) \label{eq:ndiff}\\
 & =\int_0^\infty \{R(\lambda,\delta_\pi)-R(\lambda,\delta_{\pi i})\}\pi_i(\lambda) \rd \lambda \notag\\
 &=c_n\int_{\mathbb{R}^p}\left(\left\{\psi_{\pi}^2(\|z\|^2)
 -2\psi_{\pi}(\|z\|^2)\psi_{\pi i}(\|z\|^2)\right\} \right. \notag\\
 &\qquad \left. -
 \left\{\psi_{\pi i}^2(\|z\|^2)
 -2\psi_{\pi i}(\|z\|^2)\psi_{\pi i}(\|z\|^2)\right\}\right)\|z\|^2M_1(z,\pi_i) \rd z \notag \\
 &=c_n\int_{\mathbb{R}^p}\oldiff(z;\delta_\pi,\delta_{\pi i};\pi_i) \rd z,\notag
\end{align}
where $c_n$ is given by \eqref{eq:c_n} and where
\begin{equation}\label{eq:oldiff}
 \oldiff(z;\delta_\pi,\delta_{\pi i};\pi_i)
  =\{\psi_{\pi }(\|z\|^2)-\psi_{\pi i}(\|z\|^2)\}^2\|z\|^2M_1(z,\pi_i).
\end{equation}

There are several versions of the Blyth method. 
For our purpose, the following version from \cite{Brown-1971} and \cite{Brown-Hwang-1982}
 is useful.
\begin{thm}\label{blyth-method}
 Assume that the sequence $\pi_i(\lambda)$ for $i=1,2,\dots,$ 
 satisfies
\begin{enumerate}[label= BL.\arabic*]
\item \label{BL.-1}$\pi_1(\lambda)\leq \pi_2(\lambda)\leq \dots$  for any $\lambda\geq 0$ and $\lim_{i\to\infty}\pi_i(\lambda)=\pi(\lambda)$.
 \item \label{BL.0} $ \dps \int_0^\infty \pi_i(\lambda)\rd \lambda <\infty $ for any fixed $i$.
 \item \label{BL.1} $ \dps\int_0^1 \pi_{1}(\lambda)\rd \lambda >\gamma $ for some positive
       $\gamma>0$.
 \item \label{BL.2}$\dps \lim_{i\to\infty} \ndiff(\delta_\pi,\delta_{\pi i};\pi_i)=0$.
\end{enumerate}
 Then $ \delta_\pi$ is admissible among the class $\mathcal{D}_\psi$.
\end{thm}
\begin{proof}
 See Appendix \ref{sec:ap.blyth-method}.
\end{proof} 

We consider the following assumptions on $\pi$ in addition to \eqref{eq:improper.pi.0}.
\begin{assumptionpi}\mbox{}
 \begin{enumerate}[label= A.\arabic*]
  \item \label{AA1}$\pi(\lambda)$ is differentiable.
  \item \label{AA2} (Behavior around the origin) 
For $\lambda\in[0,1]$, there exist $\alpha>-1/2$ and $\nu(\lambda)$ 
such that
\begin{align*}
\pi(\lambda)=\lambda^{\alpha}\nu(\lambda), 
\end{align*}
where
\begin{align*}
0<\nu(0)<\infty\text{ and }\lim_{\lambda\to 0} \lambda\nu'(\lambda)=0.
\end{align*}
  \item \label{AA3} (Asymptotic behavior)
       Let
$ \kappa(\lambda)=\lambda\pi'(\lambda)/\pi(\lambda)$.
Either \ref{AA3.1} or \ref{AA3.2} is assumed;
\begin{enumerate}[label= A.3.\arabic*]
 \item \label{AA3.1} $\dps -1\leq \lim_{\lambda\to\infty}\kappa(\lambda)< 0$
 \item \label{AA3.2}$\dps\lim_{\lambda\to\infty}\kappa(\lambda)= 0$.
       Further either \ref{AA3.2.1} or \ref{AA3.2.2} is assumed;
\begin{enumerate}[label= A.3.2.\arabic*]
 \item \label{AA3.2.1}$\kappa(\lambda)$ is eventually monotone increasing and approaches $0$ from below.
\item \label{AA3.2.2}$ \dps\limsup_{\lambda\to\infty}\,\{\log \lambda\}|\kappa(\lambda)|<1$.
\end{enumerate}
\end{enumerate}
	\end{enumerate}
\end{assumptionpi}
A typical prior $\pi(\lambda)$ satisfying Assumptions \ref{AA1}--\ref{AA3}, 
corresponding to a generalized \citeapos{Strawderman-1971} prior, is given by
 \begin{equation}\label{billsprior}
  \pi(\lambda;\alpha,\beta,b)=c_p\lambda^{p/2-1}
   \int_b^\infty\frac{1}{(2\pi\xi)^{p/2}}\exp\left(-\frac{\lambda}{2\xi}\right)
  (\xi-b)^\alpha(1+\xi)^\beta\rd \xi.
 \end{equation}  
Assumptions \ref{AA1}--\ref{AA3} are satisfied when
$ \{-1\leq \alpha+\beta\leq 0, \ \alpha>-1, \ b>0\}$ or
$\{-1\leq \alpha+\beta\leq 0, \alpha>-1/2,  b=0\}$.
See Appendix \ref{bill-1971} for the proof.
Note that the power prior $ \pi(\lambda)=\lambda^{\alpha}$ for $-/2<\alpha \leq 0$,
which will be considered in Section \ref{sub.sec.interesting},
corresponds to the case $\beta=0$ and $b=0$.

For a generalized prior $\pi(\lambda)$ satisfying Assumptions \ref{AA1}--\ref{AA3},
consider the sequence given by $\pi_i(\lambda)=\pi(\lambda)h_i^2(\lambda)$ where 
$h_i(\lambda)$, for $\lambda\geq 0$ and $i=1,2,\dots,$ 
is defined by
\begin{align}\label{eq:new_hir}
 h_i(\lambda)=1-\frac{\log\log(\lambda +e)}{\log\log(\lambda +e+i)},
\end{align}
and $e=\exp(1)$. 
It is clear that  $\pi_i$ satisfies \ref{BL.-1} of Theorem \ref{blyth-method}.
In Lemma \ref{lem:h_i} of  Appendix \ref{sec:assumption}, we show that
$\pi_i$ also satisfies \ref{BL.0} and \ref{BL.1} of Theorem \ref{blyth-method}.

For \ref{BL.2}, note that $\ndiff(\delta_\pi,\delta_{\pi i};\pi_i)$ given by \eqref{eq:ndiff}
is a functional of $f$ as well as $\pi$ and $\pi_i$.
Some additional assumptions on $f$ (as well as \eqref{assmp.f.1}) are required as follows;
\begin{assumptionf}\mbox{}
 \begin{enumerate}[label= F.\arabic*]
  \item \label{FF1} 
	$0<f(t)<\infty$ for any $t\geq 0$.
  \item \label{FF2} $f$ is differentiable.
  \item \label{FF3} the asymptotic behavior:
Either \ref{FF3.1} or \ref{FF3.2} is assumed;
\begin{enumerate}[label= F.3.\arabic*]
 \item \label{FF3.1} $\dps  \limsup_{t\to\infty}\, t\frac{f'(t)}{f(t)}<-\frac{p+n}{2}-2$.
 \item \label{FF3.2}  $\dps  \limsup_{t\to\infty}\, t\frac{f'(t)}{f(t)}<-\frac{p+n}{2}-3$.
 \end{enumerate}
 \end{enumerate}
\end{assumptionf}
We note that, in addition to the normal distribution,
\begin{equation*}
 f_G(t)=(2\pi)^{-(p+n)/2}\exp(-t/2),
\end{equation*}
an interesting flatter tailed class, also
satisfying Assumptions \ref{FF1}--\ref{FF3},
is given by the multivariate generalized Student $t$ with 
\begin{align*}
f(t;a,b) 
&=\int_0^\infty\frac{f_G(t/g)}{g^{(p+n)/2}}
 \frac{g^{-a/2-1}}{\Gamma(a/2)(2/b)^{a/2}}\exp\left(-\frac{b}{2g}\right)\rd g \\
 &=\frac{\Gamma((p+n+a)/2)}{(\pi b)^{(p+n)/2}\Gamma(a/2)}
(1+t/b)^{-(p+n+a)/2}.
\end{align*}
For Assumptions \ref{FF3.1} and \ref{FF3.2}, $a>4$ and $a>6$ are needed respectively.

The main result on admissibility of $\delta_\pi$ given by \eqref{eq:delta_pi_thm}
among the class $\mathcal{D}_\psi$ through
the Blyth method is as follows.
\begin{thm}\label{thm:mainmain}\mbox{}
\begin{description}
 \item[Case I]\label{thm:mainmain.1}  Assume Assumptions \ref{AA1}, \ref{AA2} and \ref{AA3.1} on $\pi$
and Assumptions \ref{FF1}, \ref{FF2} and \ref{FF3.1} on $f$.
	    Then the 
	    estimator $\delta_\pi$ given by \eqref{eq:delta_pi_thm}
	    is admissible among the class $\mathcal{D}_\psi$.
 \item[Case II]\label{thm:mainmain.2}  Assume Assumptions \ref{AA1}, \ref{AA2} and \ref{AA3.2} on $\pi$
and Assumptions \ref{FF1}, \ref{FF2} and \ref{FF3.2} on $f$.
	    Then the estimator
	    $\delta_\pi$ given by \eqref{eq:delta_pi_thm}
	    is admissible among the class $\mathcal{D}_\psi$.
\end{description}
\end{thm}
The proof of Theorem \ref{thm:mainmain}, or
essentially, equivalently the proof of \ref{BL.2},
\begin{equation*}
 \lim_{i\to \infty}\ndiff(\delta_\pi,\delta_{\pi i};\pi_i) =0
\end{equation*}
under the above Assumptions, is provided in Appendices \ref{sec:ap.pre.main}, \ref{sec:caseI} and \ref{sec:caseII}.
Prior to these sections, some preliminary needed results
on $\pi$, $f$ and $\pi_i$ are given in Appendix \ref{sec:assumption}.

\begin{remark}
 The basic idea behind the sequence $h_i$ given by \eqref{eq:new_hir} comes from the
 $h_i$ of \cite{Brown-Hwang-1982},
 \begin{align}\label{eq:BH_hir}
 h_i(\lambda)=\begin{cases}
	      1 & \lambda \leq 1  \\
1-\log \lambda/\log i & 1 \leq \lambda\leq i \\
0 & \lambda>i. 
	     \end{cases} 
\end{align} 
 A smoothed version of the above is
\begin{equation}\label{eq:new_hir_1}
 h_i(\lambda)=1-\frac{\log(\lambda+1)}{\log(\lambda+1+i)}.
\end{equation}
The sequence $h_i$ given by \eqref{eq:new_hir} is more slowly changing in both $r$ and $i$,
in order to handle priors with flatter tail than treated in \cite{Brown-Hwang-1982}.
Also, with smooth $\pi_i=\pi h_i^2$, the proofs become simpler. 
\end{remark}
\begin{remark}
Assumption \ref{AA3} is a sufficient condition for
 \begin{equation}\label{eq:known_variance_0}
  \int_1^\infty \frac{\rd \lambda}{\lambda \pi(\lambda)}   =\infty \ 
   \Leftrightarrow \
   \int_1^\infty\frac   {\rd \lambda}{\lambda^{p/2} \bar{\pi}(\lambda)}=\infty,
 \end{equation}
 which is related to admissibility in the known variance case as follows.
 \cite{Maruyama-2009} showed that, in the problem of estimating $\mu$ of $X\sim N_p(\mu,I)$,
 regularly varying priors $ g(\|\mu\|^2)$ with
\begin{equation}\label{eq:known_variance_1}
 \int_1^\infty \frac{\rd \lambda}{\lambda^{p/2}g(\lambda)}=\infty
\end{equation}
 lead to admissibility, that is, the (generalized) Bayes estimator
       \begin{equation*}
X+\nabla \log m_g(\|X\|^2)
       \end{equation*}
 where
 \begin{equation}
  m_g(\|x\|^2)=\frac{1}{(2\pi)^{p/2}}\int \exp(-\|x-\mu\|^2/2)g(\|\mu\|^2)\rd \mu
 \end{equation}
 is admissible.
 As \cite{Maruyama-2009} pointed out, the sufficient condition \eqref{eq:known_variance_1},
 which depends directly on the prior $g(\|\mu\|^2)$, 
 is closely related to \citeapos{Brown-1971} sufficient condition for admissibility
 \begin{equation*}
   \int_1^\infty \frac{\rd r}{r^{p/2}m_g(r)}=\infty,
 \end{equation*}
which depends on the marginal distribution and only indirectly on the prior.
Note also that Assumption \ref{AA3} is tight for the non-integrability of \eqref{eq:known_variance_0}, in the
 sense that,
among the class $\pi(\lambda)\approx\{\log\lambda\}^b$ with $b\in\mathbb{R}$,
\begin{equation*}
  \int_1^\infty \frac{\rd \lambda}{\lambda \{\log\lambda\}^{1-\epsilon}}=\infty, \ 
  \int_1^\infty \frac{\rd \lambda}{\lambda \log\lambda}=\infty, \
  \int_1^\infty \frac{\rd \lambda}{\lambda \{\log\lambda\}^{1+\epsilon}}<\infty
\end{equation*}
 where in the first expression $\{\log\lambda\}^{1-\epsilon}$ satisfies  Assumption \ref{AA3},
 and in the second, $\log \lambda$ does not satisfy  Assumption \ref{AA3}.
 Actually, in the third case, $ \{\log\lambda\}^{1+\epsilon}$,
 the corresponding Bayes equivariant estimator is inadmissible as shown
 in \cite{Maruyama-Strawderman-2017}.
\end{remark}

  \subsection{Some interesting cases}
\label{sub.sec.interesting}
Here we present three interesting special cases of our main general theorem.
\begin{corollary}\label{thm:interesting.1}
Assume 
Assumptions \ref{FF1}, \ref{FF2} and \ref{FF3.2} on $f$.
 \begin{enumerate}
  \item \label{thm:interesting.1.1}
	Then $\delta_\pi$ with $\pi\equiv 1$, or equivalently the generalized Bayes estimator under
the prior on $(\theta,\eta)$ given by 
	\begin{align*}
  \eta^{-1}\eta^{p/2}\left\{\eta\|\theta\|^2\right\}^{(2-p)/2},
	\end{align*}
is admissible among the class $\mathcal{D}_\psi$. 
  \item \label{thm:interesting.1.2}
	The form of the generalized Bayes estimator does not depend on $f$ and is given by
$\left\{1-\psi_0(W)\right\}X$ where $W=\|X\|^2/S$ and
	\begin{align*}
 \psi_0(w)=
  \frac{\int_0^1 t^{p/2-1}(1+wt)^{-(p+n)/2-1}\rd t}{\int_0^1 t^{p/2-2}(1+wt)^{-(p+n)/2-1}\rd t}.
	\end{align*}
  \item \label{thm:interesting.1.3}
	This estimator is minimax simultaneously for all such $f$.
  \item \label{thm:interesting.1.4}
	This estimator dominates the James--Stein estimator
		\begin{align*}
 \left(1-\frac{p-2}{n+2}\frac{S}{\|X\|^2}\right)X
\end{align*}
if $f$ is nonincreasing.
 \end{enumerate}
\end{corollary}
\begin{proof}
 For Part \ref{thm:interesting.1.1}, Assumptions \ref{AA1}, \ref{AA2} and \ref{AA3.2} are satisfied by
 $\pi(\lambda)\equiv 1$.
 Parts \ref{thm:interesting.1.2} and \ref{thm:interesting.1.4} are both shown by \cite{Maruyama-2003b}.
Part \ref{thm:interesting.1.3} is shown by \cite{Cellier-et-al-1989}.
\end{proof}

\begin{corollary}\label{thm:interesting.2.main}
Assume 
Assumptions \ref{FF1}, \ref{FF2} and \ref{FF3.1} on $f$. Let $\alpha\in(-1/2,0)$.
 \begin{enumerate}
  \item \label{thm:interesting.2.1}
	Then $\delta_\pi$ with $\pi(\lambda)=\lambda^{\alpha} $,
	or equivalently the generalized Bayes estimator under
the prior on $(\theta,\eta)$ given by 
 \begin{align*}
  \eta^{-1}\eta^{p/2}\left\{\eta\|\theta\|^2\right\}^{\alpha+(2-p)/2},
 \end{align*}
	is admissible among the class $\mathcal{D}_\psi$. 
  \item \label{thm:interesting.2.2}
	The form of the estimator does not depend on $f$ and is given by
$\left\{1-\psi_\alpha(W)\right\}X$ where $W=\|X\|^2/S$ and
	\begin{equation}\label{psialphageneral}
 \psi_\alpha(w)=
  \frac{\int_0^1 t^{p/2-\alpha-1}(1-t)^{\alpha}(1+wt)^{-(p+n)/2-1}\rd t}
  {\int_0^1 t^{p/2-\alpha-2}(1-t)^{\alpha}(1+wt)^{-(p+n)/2-1}\rd t}.
	\end{equation}
  \item \label{thm:interesting.2.3}
	This estimator is minimax when
\begin{equation*}
 -\left(5+\frac{2}{p-2}+\frac{3p}{n+2}\right)^{-1} \leq \alpha <0.
\end{equation*}
 \end{enumerate}
\end{corollary}
\begin{proof}
 For Part \ref{thm:interesting.1.1}, Assumptions \ref{AA1}, \ref{AA2} and \ref{AA3.1} are satisfied by
 $\pi(\lambda)=\lambda^\alpha$ for $\alpha\in(-1/2,0)$.
 Part \ref{thm:interesting.2.2} is shown by \cite{Maruyama-2003b}.
 For Part \ref{thm:interesting.2.3}, see \cite{Maruyama-Strawderman-2009}
 and Appendix \ref{sec.ap.alam}.
\end{proof}

The following corollary relates to the so-called ``simple Bayes estimators'' from \cite{Maru-Straw-2005}.
\begin{corollary}\label{thm:interesting.3.main}
 Assume $f$ is Gaussian. Then
 the simple Bayes estimator
\begin{align*}
 \left(1-\frac{a}{(a+1)(b+1)+\|X\|^2/S}\right)X
\end{align*}
with $a\geq (p-2)/(n+2)$ and $b\geq 0$ is admissible among the class $\mathcal{D}_\psi$.
Furthermore, the estimator with $(p-2)/(n+2)\leq a\leq 2(p-2)/(n+2)$ is minimax.
\end{corollary}
\begin{proof}
 The estimator is (generalized) Bayes equivariant estimator with respect to
 $\pi(\lambda;\alpha,\beta,b)$ given by \eqref{billsprior}
 with $\beta=-n/2$ and $\alpha=(p+n)/\{2(a+1)\}-1$. 
 See \cite{Maru-Straw-2005} and Appendix \ref{sec.ap.2005}.
\end{proof}

\section{Canonical form of the regression setup}
\label{sec.regression}
 Suppose a linear regression model is used to relate $y$ to the
$p$ predictors $z_1, \dots, z_p$,
\begin{equation} \label{full-model}
y = \alpha 1_m+Z\beta + \eta^{-1/2}\epsilon
\end{equation}
where $\alpha$ is an unknown intercept parameter,
$1_m$ is an $m\times 1$ vector of ones,
$Z=(z_1,\dots, z_p)$ is an $m \times p$ design matrix,
and $\beta$ is a $p \times 1$ vector of unknown
regression coefficients.
In the error term, $\eta$ is an unknown scalar and
 $\epsilon=(\epsilon_1,\dots,\epsilon_m)^\T$ 
 has a spherically symmetric distribution,
\begin{equation}\label{bep_sim_f}
\epsilon \sim \tilde{f}(\|\epsilon\|^2)
\end{equation}
where $\tilde{f}(\cdot)$ is the probability density,
$E[\epsilon]=0_m$, 
and $\mbox{Var}[\epsilon]=I_m$.
Hence the density of $y$ is
\begin{equation}
 y\sim \eta^{m/2}\tilde{f}(\eta\|y-\alpha 1_m-Z\beta\|^2),
\end{equation}
where $\tilde{f}$ satisfies 
\begin{equation*}
 \int_{\mathbb{R}^m}\tilde{f}(\|v\|^2) \rd v=1 
\end{equation*}
for $v=(v_1,\dots,v_m)^\T\in\mathbb{R}^m$.
We assume that the columns of $Z$ have been centered so that 
$z_i^\T 1_m = 0$ for $1 \leq i \leq p$.
We also assume that $m > p+1$ and $\{z_1,\dots,z_p\}$ are linearly independent, 
which implies that 
\begin{equation*}
\mbox{rank} \ Z=p.
\end{equation*}

Let $Q$ be an $m\times m$ orthogonal matrix of the form
\begin{equation*}
 Q =(1_m/\sqrt{m}, Z(Z^\T Z)^{-1/2}, W)
\end{equation*}
where $W$ is $m\times (m-p-1)$ matrix which satisfies $ W^\T 1_m=0$, $ W^\T Z=0$
and $W^\T W=I_{m-p-1}$.
Also let $x=(Z^\T Z)^{-1/2}Z^\T y=(Z^\T Z)^{1/2}\hat{\beta}_{\mathrm{LSE}}\in\mathbb{R}^p$
where $\hat{\beta}_{\mathrm{LSE}}=(Z^\T Z)^{-1}Z^\T y$.

Let
\begin{equation*}
 Q^\T y=(\sqrt{m}\bar{y}, x^\T, u^\T)^\T
\end{equation*}
where $u=W^\T y\in\mathbb{R}^{m-p-1}$.
Then $ (\sqrt{m}\bar{y}, x, u) $ are sufficient and
the joint density of $ (\sqrt{m}\bar{y}, x, u) $ 
is
\begin{equation*}
 \eta^{m/2}\tilde{f}(\eta\{m(\bar{y}-\alpha)^2+\|x-\theta\|^2+\|u\|^2\})
\end{equation*}
where $\theta=(Z^\T Z)^{1/2}\beta$.
Further the marginal density of $ (x, u) $ is
\begin{equation*}
 \eta^{(m-1)/2}f(\eta\{\|x-\theta\|^2+\|u\|^2\}),
\end{equation*}
which we are considering in this paper, where $m-1=p+n$ and 
\begin{equation*}
 f(t)=\int_{-\infty}^\infty \tilde{f}(v^2+t)\rd v.
\end{equation*}
Note that the loss function $\eta\|\delta-\theta\|^2$ corresponds to
so-called ``predictive loss'' $\eta\|Z\hat{\beta}-Z\beta\|^2$ for estimation of
the regression coefficient vector $\beta$.

In the equivariant estimator $\delta_\psi$ of $\theta$
\begin{equation*}
 \left\{1-\psi(\|x\|^2/s)\right\}x,
\end{equation*}
$\|x\|^2/s$ is $R^2/(1-R^2)$ in the regression context where $R^2$ is the coefficient of
determination. 
It is natural to make use of $R^2$ for shrinkage since small $R^2$
corresponds to less reliability of the least squares estimator of $\beta$.
We note that the corresponding ``simple Bayes estimator'' for regression coefficient $\beta$
 is rewritten as
\begin{equation*}
 \left(1-\frac{a}{(a+1)(b+1)+R^2/(1-R^2)}\right)\hat{\beta}_{\mathrm{LSE}}
\end{equation*}
and has a shrinkage factor which is increasing in $R^2$.

In the equivariant estimator $\delta_\xi=\left\{1-\xi(x/\sqrt{s})\right\}x\in \mathcal{D}_\xi$,
\begin{equation}
 \frac{x}{\sqrt{s}}=\frac{(Z^\T Z)^{1/2}\hat{\beta}_{\mathrm{LSE}}}{\sqrt{m-p-1}\hat{\sigma}}
  =\frac{(Z^\T Z)^{1/2}}{\sqrt{m-p-1}}\frac{\hat{\beta}_{\mathrm{LSE}}}{\hat{\sigma}}
\end{equation}
where $\hat{\sigma}=\sqrt{s/(m-p-1)}$ and $ \hat{\beta}_{\mathrm{LSE}}/\hat{\sigma}$ is
a vector of the $t$-values.

Hence the restriction to $\mathcal{D}_\psi $ or $\mathcal{D}_\xi$
is quite natural in regression context. 
The minimaxity and admissibility results of Sections \ref{sec:equiv} and \ref{sec:equiv.blyth} provide some
guidance as to reasonable shrinkage estimators in the regression context.

\section{Concluding remarks}
\label{sec.cr}
We have established admissibility of certain generalized Bayes estimators
within the class of equivariant estimators, of the mean vector for a spherically
symmetric distribution with unknown scale under invariant loss.
In some cases, we establish simultaneous minimaxity and, equivariant admissibility
 for broader classes of sampling distributions. In the Gaussian case we establish
 admissibility within the equivariant estimators
 of a class of generalized Bayes minimax estimators of a particularly simple form.
 We have also investigated similar issues in the setting of a general linear regression
 model with intercept and spherically symmetric error distribution. In this setting,
 the shrinkage factor of equivariant estimators of the regression coefficients depends
 on the coefficient of determination.
 
\appendix
\section{Proof of Theorem \ref{thm:stein.1956}}
\label{sec:ap.stein.1956}
 Suppose the estimator $\delta_\xi(X,S)\in\mathcal{D}_\xi$
is strictly better than the estimator $\delta_\psi\in\mathcal{D}_\psi$,  that is,
\begin{align}\label{xi_psi_ineq}
 E\left[\eta\left\|\delta_\xi(X,S)-\theta\right\|^2\right]
 \leq  E\left[\eta\left\|\delta_\psi(X,S)-\theta\right\|^2\right]
\end{align}
 for all $\eta^{1/2}\theta\in\mathbb{R}^p$ with strict inequality for some value.
Because of the continuity of $\delta_\xi(X,S)$ and $\delta_\psi(X,S)$,
 strict inequality will hold for
  $\eta^{1/2}\theta\in\mathbb{R}^p$ in some nonempty open set
$U\subset\mathbb{R}^p$.
The inequality \eqref{xi_psi_ineq} will remain true
if $\delta_\xi(X,S)$ is replaced by $ \Gamma\delta_\xi(\Gamma^{-1}X,S)$
with $\Gamma$ orthogonal, since
\begin{align*}
 E\left[\eta\left\|\Gamma\delta_\xi(\Gamma^{-1}X,S)-\theta\right\|^2\right] 
= E\left[\eta\left\|\delta_\xi(\Gamma^{-1}X,S)-\Gamma^{-1}\theta\right\|^2\right].
\end{align*}
 Thus, for fixed  $\eta^{1/2}\theta\in U\subset\mathbb{R}^p$,
the set of $\Gamma$ for which
\begin{align*}
 E\left[\eta\left\|\Gamma\delta_\xi(\Gamma^{-1}X,S)-\theta\right\|^2\right] 
<E\left[\eta\left\|\delta_\psi(X,S)-\theta\right\|^2\right]
\end{align*}
will be a nonempty open set.
Let $\mu$ be the invariant probability measure on $\mathcal{O}(p)$
which assigns strictly positive measure to any nonempty open set
(for the existence of such a measure, see Chapter 2 of \cite{Weil-1940}).
Then the weighted estimator
\begin{align*}
\delta_{\xi\star}= \int_{\mathcal{O}(p)} \Gamma\delta_\xi(\Gamma^{-1}X,S)\rd \mu(\Gamma)
\end{align*}
 is a member of the class $\mathcal{D}_\psi$,
 and because of the convexity of the loss function in $\delta$, we have 
\begin{align*}
 E\left[\eta\left\|\delta_{\xi\star}(X,S)-\theta\right\|^2\right] 
 &\leq \int E\left[\eta\left\|\Gamma\delta_\xi(\Gamma^{-1}X,S)-\theta\right\|^2\right]\rd \mu(\Gamma) \\
 &\leq E\left[\eta\left\|\delta_\psi(X,S)-\theta\right\|^2\right]
\end{align*}
with strict inequality for $\eta^{1/2}\theta\in U$.
This implies that $\delta_\psi(X,S)$ is not admissible among $\mathcal{D}_\psi$ as assumed 
and hence completes the proof.

\section{Proof of Theorem \ref{thm:proper.admissible}}
\label{sec:ap.proper.admissible}

[Parts \ref{thm:proper.admissible.0} and \ref{thm:proper.admissible.1}]\mbox{}
The Bayes equivariant risk given by \eqref{eq:eBrisk} is rewritten as
\begin{equation}\label{eq:eBrisk.1}
\begin{split}
B(\delta_\psi, \pi)&=\int_{\mathbb{R}^p} \tilde{R}(\|\mu\|^2,\delta_{\psi})\bar{\pi}(\|\mu\|^2) \rd \mu \\
 &=\int_{\mathbb{R}^p} \tilde{R}(\eta\|\theta\|^2,\delta_{\psi})\eta^{p/2}\bar{\pi}(\eta\|\theta\|^2) \rd \theta \\
 &=\int_{\mathbb{R}^p} R(\theta,\eta,\delta_{\psi})\eta^{p/2}\bar{\pi}(\eta\|\theta\|^2) \rd \theta,
\end{split}
\end{equation}
where the third equality follows from \eqref{eq:thetaeta}.
Further $B(\delta_\psi, \pi) $ given by \eqref{eq:eBrisk.1} is expanded as
\begin{equation}\label{eq:expand.Brisk}
 \begin{split}
 B(\delta_\psi, \pi)
 &=\int_{\mathbb{R}^p} E\left[\eta\|X\|^2\psi^2(\|X\|^2/S)\right]\eta^{p/2}\bar{\pi}(\eta\|\theta\|^2) \rd \theta \\
 &\qquad -2\int_{\mathbb{R}^p} E\left[\eta\|X\|^2\psi(\|X\|^2/S)\right]
 \eta^{p/2}\bar{\pi}(\eta\|\theta\|^2) \rd \theta \\
 &\qquad +2\int_{\mathbb{R}^p} E\left[\eta\psi(\|X\|^2/S)X^\T\theta\right]
 \eta^{p/2}\bar{\pi}(\eta\|\theta\|^2) \rd \theta \\
&\qquad + \int_{\mathbb{R}^p}E\left[\eta\|X-\theta\|^2\right]\eta^{p/2}\bar{\pi}(\eta\|\theta\|^2) \rd \theta. 
\end{split}
\end{equation}
Note that, by \eqref{assmp.f.1} and the propriety of the prior given by \eqref{eq:propriety.0},
the third term is equal to $p$, that is,
\begin{equation}\label{eq:equiv.third.term}
\int_{\mathbb{R}^p}E\left[\eta\|X-\theta\|^2\right]\eta^{p/2}\bar{\pi}(\eta\|\theta\|^2) \rd \theta 
=\int_{\mathbb{R}^p}p\bar{\pi}(\|\mu\|^2) \rd \mu   =p.
\end{equation}

The first and second terms of \eqref{eq:expand.Brisk} with $ \psi^j(\|X\|^2/S)$ for $j=2,1$ respectively, are rewritten as
\begin{align}
 & \int_{\mathbb{R}^p} E\left[\eta\|X\|^2\psi^j(\|X\|^2/S)\right]
 \eta^{p/2}\bar{\pi}(\eta\|\theta\|^2) \rd \theta \label{eq:equiv.first.term} \\
 &= c_n \iiint \eta\|x\|^2\psi^j(\|x\|^2/s)
 \eta^{(2p+n)/2}s^{n/2-1}f(\eta\{\|x-\theta\|^2+s\}) \notag\\ &\qquad \times 
 \bar{\pi}(\eta\|\theta\|^2) \rd \theta  \rd x  \rd s \notag\\
 &= c_n \iiint \eta s\|z\|^2\psi^j(\|z\|^2) 
 \eta^{(2p+n)/2}s^{(p+n)/2-1}f(\eta\{\|\sqrt{s}z-\theta\|^2+s\}) \notag\\ &\qquad \times 
\bar{\pi}(\eta\|\theta\|^2) \rd \theta  \rd z  \rd s \quad (z=x/\sqrt{s}, \ J=s^{p/2})\notag\\
 &= c_n \iiint \eta s\|z\|^2\psi^j(\|z\|^2) 
 \eta^{(2p+n)/2}s^{(2p+n)/2-1}f(s\eta\{\|z-\theta_*\|^2+1\}) \notag\\ &\qquad \times 
 \bar{\pi}(\eta s\|\theta_*\|^2) \rd \theta_*  \rd z  \rd s \quad (\theta_*=\theta/\sqrt{s},
 \ J=s^{p/2})\notag\\
 &= c_n \iiint \|z\|^2\psi^j(\|z\|^2) 
 \eta_*^{(2p+n)/2}f(\eta_*\{\|z-\theta_*\|^2+1\}) \notag \\ &\qquad \times 
 \bar{\pi}(\eta_*\|\theta_*\|^2) \rd \theta_*  \rd z  \rd \eta_* \quad (\eta_*=\eta s, \ J=1/\eta)\notag\\
 &= c_n\int_{\mathbb{R}^p} \|z\|^2\psi^j(\|z\|^2) M_1(z,\pi) \rd z, \notag
\end{align}
where $c_n$ is given by \eqref{eq:c_n}, $z=x/\sqrt{s}$, $J$ is the Jacobian, and
\begin{equation}\label{eq:def:M_1}
 M_1(z,\pi)=
 \iint \eta^{(2p+n)/2}f(\eta\{\|z-\theta\|^2+1\})
\bar{\pi}(\eta\|\theta\|^2) \rd \theta   \rd \eta .
\end{equation}

Similarly, the third term of \eqref{eq:expand.Brisk} is rewritten as
\begin{align}
 & \int_{\mathbb{R}^p} E\left[\eta\psi(\|X\|^2/S)X^\T\theta\right]\eta^{p/2}\bar{\pi}(\eta\|\theta\|^2) \rd \theta \label{eq:equiv.second.term}\\
 &= c_n\iiint \eta\psi^2(\|x\|^2/s)x^\T\theta
  \eta^{(2p+n)/2}s^{n/2-1}f(\eta\{\|x-\theta\|^2+s\}) \notag\\
&\qquad\times \bar{\pi}(\eta\|\theta\|^2) \rd \theta  \rd x  \rd s \notag\\
 &= c_n\iiint \eta \psi(\|z\|^2) \sqrt{s}z^\T\theta 
  \eta^{(2p+n)/2}s^{(p+n)/2-1}f(\eta\{\|\sqrt{s}z-\theta\|^2+s\}) \notag\\ &\qquad \times 
\bar{\pi}(\eta\|\theta\|^2) \rd \theta  \rd z  \rd s \quad (z=x/\sqrt{s}, \ J=s^{p/2})\notag\\
 &= c_n\iiint \eta s\psi(\|z\|^2) z^\T\theta_*
 \eta^{(2p+n)/2}s^{(2p+n)/2-1}f(s\eta\{\|z-\theta_*\|^2+1\}) \notag\\ &\qquad \times 
 \bar{\pi}(\eta s \|\theta_*\|^2) \rd \theta_*  \rd z  \rd s
 \quad (\theta_*=\theta/\sqrt{s}, \ J=s^{p/2})\notag\\
 &= c_n\iiint \psi(\|z\|^2)z^\T\theta_* 
  \eta_*^{(2p+n)/2}f(\eta_*\{\|z-\theta_*\|^2+1\})  \notag\\ &\qquad \times 
\bar{\pi}(\eta_*\|\theta_*\|^2) \rd \theta_*  \rd z  \rd \eta_* \quad (\eta_*=\eta s, \ J=1/\eta) \notag\\
 &= c_n\int_{\mathbb{R}^p} \psi(\|z\|^2)z^\T M_2(z,\pi) \rd z,\notag
\end{align}
where
\begin{equation}\label{eq:def:M_2}
 M_2(z,\pi)=
\iint \theta
  \eta^{(2p+n)/2} f(\eta\{\|z-\theta\|^2+1\}) \bar{\pi}(\eta\|\theta\|^2) \rd \theta  \rd \eta.
\end{equation}
Hence, by \eqref{eq:equiv.third.term}, \eqref{eq:equiv.first.term} and \eqref{eq:equiv.second.term},
 we have
\begin{equation}\label{eq:Bayesequivrisk}
 \begin{split}
 B(\delta_\psi, \pi) &=
 c_n\int_{\mathbb{R}^p} \left\{\psi^2(\|z\|^2) \|z\|^2 M_1(z,\pi) \right. \\
&\qquad \left. -2\psi(\|z\|^2)\{\|z\|^2 M_1(z,\pi)-z^\T M_2(z,\pi)\}\right\} \rd z +p. 
\end{split}
\end{equation}
Then the Bayes equivariant solution or minimizer of $ B(\delta_\psi, \pi)$ is
\begin{equation}\label{eq:eq_Bayes_sol}
 \psi_\pi(\|z\|^2)=\argmin_\psi\left(B(\delta_\psi, \pi)
 \right) 
 =1-\frac{z^\T M_2(z,\pi)}{\|z\|^2 M_1(z,\pi)}
\end{equation}
and hence the corresponding Bayes equivariant estimator is
\begin{align}\label{eq:delta_pi}
\delta_\pi= \frac{Z^\T M_2(Z,\pi)}{\|Z\|^2 M_1(Z,\pi)}X,
\end{align}
where $Z=X/\sqrt{S}$.

[Part \ref{thm:proper.admissible.2}]\mbox{}
The generalized Bayes estimator of $\theta$ with respect to the density on $(\theta,\eta)$,
\begin{equation*}
 \eta^\nu\eta^{p/2}g(\eta\|\theta\|^2)
\end{equation*}
is given by
\begin{align}
 \delta_{g,\nu}&
=\frac{E[\eta\theta\mid x,s]}{E[\eta\mid x,s]} \notag\\
& =\frac{\iint \eta\theta c_n\eta^{(p+n)/2}s^{n/2-1}f(\eta\{\|x-\theta\|^2+s\})\eta^\nu\eta^{p/2}g(\eta\|\theta\|^2) \rd \theta  \rd \eta}
 {\iint \eta c_n\eta^{(p+n)/2}s^{n/2-1}f(\eta\{\|x-\theta\|^2+s\})\eta^\nu\eta^{p/2}g(\eta\|\theta\|^2) \rd \theta  \rd \eta} \notag \\
 &=\frac{\iint \theta \eta^{(2p+n)/2+\nu+1}f(\eta\{\|x-\theta\|^2+s\})g(\eta\|\theta\|^2) \rd \theta  \rd \eta}
 {\iint  \eta^{(2p+n)/2+\nu+1}f(\eta\{\|x-\theta\|^2+s\})g(\eta\|\theta\|^2) \rd \theta  \rd \eta}.
 \label{eq:ap.gbgnu}
\end{align}
By change of variables $\theta_*=\theta/\sqrt{s}$ and $\eta_*=s\eta$, we have
 \begin{align*}
  \delta_{g,\nu} =\sqrt{s}\frac{\iint \theta_* \eta_*^{(2p+n)/2+\nu+1}f(\eta_*\{\|x/\sqrt{s}-\theta_*\|^2+1\})
  g(\eta_*\|\theta_*\|^2) \rd \theta_*  \rd \eta_*}
 {\iint \eta_*^{(2p+n)/2+\nu+1}f(\eta_*\{\|x/\sqrt{s}-\theta_*\|^2+1\})g(\eta_*\|\theta_*\|^2) \rd \theta_*  \rd \eta_*}. 
\end{align*}
Comparing $\delta_{g,\nu}$ with $\delta_\pi$ given by \eqref{eq:delta_pi},
we see that $\delta_{g,\nu}$ with $\nu=-1$ is 
\begin{equation*}
\delta_{g,-1}  =\sqrt{s}\frac{M_2(z,g)}{M_1(z,g)} 
=\sqrt{s}\frac{z z^\T M_2(z,g)}{\|z\|^2 M_1(z,g)} 
 =\frac{z^\T M_2(z,g)}{\|z\|^2 M_1(z,g)}x.
\end{equation*}
The second equality follows since $M_2(z,g)$ is proportional to $z$ and
the length of $ M_2(z,g)$ is $z^\T M_2(z,g)/\|z\|$.

[Part \ref{thm:proper.admissible.3}]\mbox{}
Since the  quadratic loss function is strictly convex, the Bayes solution is unique,
and hence admissibility within $\mathcal{D}_\psi$ follows.

\section{Proof that $\delta_{g,\nu}\in\mathcal{D}_\psi$}
\label{sec:ap.nu.neq.-1}
As in \eqref{eq:ap.gbgnu},
the generalized Bayes estimator of $\theta$ with respect to 
$ \eta^\nu\eta^{p/2}g(\eta\|\theta\|^2)$ is given by
\begin{align*}
  \delta_{g,\nu}(x,s)=\frac{\iint \theta \eta^{(2p+n)/2+\nu+1}f(\eta\{\|x-\theta\|^2+s\})g(\eta\|\theta\|^2) \rd \theta  \rd \eta}
 {\iint  \eta^{(2p+n)/2+\nu+1}f(\eta\{\|x-\theta\|^2+s\})g(\eta\|\theta\|^2) \rd \theta  \rd \eta}.
\end{align*}
The estimator $  \delta_{g,\nu}(x,s)$ with $x=\gamma\Gamma x$ and $s=\gamma^2s$ is
\begin{align*}
 \delta_{g,\nu}(\gamma\Gamma x,\gamma^2s)
 =
\frac{\iint \theta \eta^{(2p+n)/2+\nu+1}f(\eta\{\|\gamma\Gamma x-\theta\|^2+\gamma^2 s\})g(\eta\|\theta\|^2) \rd \theta  \rd \eta}
 {\iint  \eta^{(2p+n)/2+\nu+1}f(\eta\{\|\gamma\Gamma x-\theta\|^2+\gamma^2 s\})g(\eta\|\theta\|^2) \rd \theta  \rd \eta}
\end{align*}
and, by change of variables $\theta=\gamma\Gamma \theta_*$ and $\eta_*=\gamma^2\eta$,
is rewritten as
\begin{align*}
 \delta_{g,\nu}(\gamma\Gamma x,\gamma^2s)
&=\gamma\Gamma\frac{\iint \theta_* \eta_*^{(p+n)/2+\nu+1}
 f(\eta_*\{\| x-\theta_*\|^2+ s\})g(\eta_*\|\theta_*\|^2) \rd \theta_*  \rd \eta_*}
 {\iint \eta_*^{(p+n)/2+\nu+1}f(\eta_*\{\| x-\theta_*\|^2+ s\})
 g(\eta_*\|\theta_*\|^2) \rd \theta_*  \rd \eta_*} \\
 &=\gamma\Gamma  \delta_{g,\nu}(x,s).
\end{align*}
Hence $\delta_{g,\nu}\in\mathcal{D}_\psi$.

\section{Proof of Theorem \ref{blyth-method}}
\label{sec:ap.blyth-method}
 Suppose that $ \delta_\pi\in\mathcal{D}_\psi$ is inadmissible among the class $\mathcal{D}_\psi$
 and hence  $\delta'\in\mathcal{D}_\psi$ satisfies 
$ \tilde{R}(\lambda,\delta') \leq \tilde{R}(\lambda,\delta_\pi)$ for all $\lambda$
with strict inequality for some $\lambda$.
 Let $\delta''=(\delta_\pi+\delta')/2$. Clearly $\delta''$ is also a member of
 $\mathcal{D}_\psi$.
Then, using Jensen's inequality, we have
\begin{align*}
 \tilde{R}(\lambda,\delta'')
&=E\left[\eta\|\delta''-\theta\|^2\right] \\
&<(1/2)E\left[\eta\|\delta'-\theta\|^2\right] +(1/2)E\left[\eta\|\delta_\pi -\theta\|^2\right] \\
 &= \frac{1}{2}\left\{\tilde{R}(\lambda,\delta')+\tilde{R}(\lambda,\delta_\pi)\right\} \\
 &\leq  \tilde{R}(\lambda,\delta_\pi),
\end{align*}
for any $\lambda$. 
Since $  \tilde{R}(\lambda,\delta'')$ and $  \tilde{R}(\lambda,\delta_\pi)$ are both continuous
 functions of $\lambda$, there exists an $ \epsilon>0$ such that
$ \tilde{R}(\lambda,\delta'')< \tilde{R}(\lambda,\delta_\pi)-\epsilon$ for 
$ 0\leq \lambda\leq 1$. 
Then 
\begin{align*}
 \ndiff(\delta_\pi,\delta_{\pi i};\pi_i)
 &= \int_0^\infty \left\{\tilde{R}(\lambda,\delta_\pi)-\tilde{R}(\lambda,\delta_{\pi i})\right\}
 \pi_i(\lambda)\rd \lambda \\
 &\geq  \int_0^\infty
\left\{\tilde{R}(\lambda,\delta_\pi)-\tilde{R}(\lambda,\delta'')\right\}\pi_i(\lambda)\rd \lambda \\
 & \geq \int_0^1
\left\{\tilde{R}(\lambda,\delta_\pi)-\tilde{R}(\lambda,\delta'')\right\}\pi_1(\lambda)\rd \lambda \\
& \geq \epsilon \gamma>0,
\end{align*}
which contradicts $\ndiff(\delta_\pi,\delta_{\pi i};\pi_i) \to 0$ as $i\to\infty$.

\section{Preliminary results on $\pi$, $\pi_i$ and $f$}
\label{sec:assumption}
  \subsection{Preliminary results on $\pi$}
\begin{lemma}\label{lem:pi}
 \begin{enumerate}
  \item \label{lem:pi:1}
	Under Assumptions \ref{AA1}--\ref{AA3},
\begin{align*}
 \sup_{\lambda\in\mathbb{R}_+}\lambda\frac{|\pi'(\lambda)|}{\pi(\lambda)}
\end{align*}
is bounded.
  \item \label{lem:pi:2}
	Under Assumption \ref{AA2}, $\dps\int_0^1\frac{\pi(\lambda)}{\lambda^{1/2}}\rd \lambda<\infty$.
  \item \label{lem:pi:3}
	Under Assumption \ref{AA2} with $\alpha>0$, $\dps  \int_0^1\frac{\pi(\lambda)}{\lambda}\rd \lambda<\infty$.
  \item \label{lem:pi:4}
	Under Assumption \ref{AA3.1},
	$\dps \int_1^\infty \frac{\pi(\lambda)}{\lambda}\rd \lambda<\infty$.
  \item \label{lem:pi:5}
	Under Assumption \ref{AA3}, $\dps  \int_1^\infty \frac{\pi(\lambda)}{\lambda^2}\rd \lambda<\infty$.
  \item \label{lem:pi:6}
	If $\lim_{\lambda\to\infty}\lambda\pi'(\lambda)/\pi(\lambda)<-1$,
	$\dps \int_1^\infty \pi(\lambda)\rd \lambda<\infty$.
  \item \label{lem:pi:7}
	Under Assumption \ref{AA3}, there exist $\epsilon\in(0,1)$ and $\lambda_*>\exp(1) $ such that
$\pi(\lambda)/\{\log \lambda\}^{1-\epsilon}$
for $\lambda\geq \lambda_*$ is bounded from above.
  \item \label{lem:pi:8} Under Assumption \ref{AA3.2},
	$\dps \int_1^\infty \frac{\pi(\lambda)}{\lambda}\kappa^2(\lambda)\rd \lambda<\infty$.
 \end{enumerate}
\end{lemma}
 \begin{proof}\mbox{}
[Part \ref{lem:pi:1}] This follows from Assumptions \ref{AA1}--\ref{AA3}
  in a straightforward way.
  
[Part \ref{lem:pi:2}] We have
\begin{equation}\label{eq:lem:pi:2.1}
\int_0^1\frac{\pi(\lambda)}{\lambda^{1/2}}\rd \lambda\leq
 \sup_{\lambda\in[0,1]}\nu(\lambda)\int_0^1\lambda^{1/2+\alpha-1}\rd \lambda
 =\sup_{\lambda\in[0,1]}\nu(\lambda)\frac{1}{1/2+\alpha}<\infty.
\end{equation}

  [Part \ref{lem:pi:3}] 
  As in \eqref{eq:lem:pi:2.1} of Part \ref{lem:pi:2}, we have
 \begin{equation*}
  \int_0^1\frac{\pi(\lambda)}{\lambda}\rd \lambda\leq\sup_{\lambda\in[0,1]}\nu(\lambda)\frac{1}{\alpha}<\infty.
 \end{equation*}
[Part \ref{lem:pi:4}]
 By Assumption \ref{AA3.1}, there exist $\epsilon>0$ and $ \lambda_1>0$ such that
\begin{align*}
 \lambda\frac{\pi'(\lambda)}{\pi(\lambda)}\leq -\epsilon,
\end{align*}
for all $\lambda\geq \lambda_1$ and hence
we have
\begin{equation}\label{eq:korene.0}
 \int_{\lambda_1}^\lambda \frac{\pi'(s)}{\pi(s)}\rd s \leq -\epsilon \int_{\lambda_1}^\lambda \frac{1}{s}\rd s
\, \Leftrightarrow \,
  \log \frac{\pi(\lambda)}{\pi(\lambda_1)}
 \leq  -\epsilon\log\frac{\lambda}{\lambda_1}
\end{equation}
for $\lambda\geq \lambda_1$, which implies that
\begin{equation}\label{eq:korene}
 \pi(\lambda)\leq \frac{\pi(\lambda_1)}{\lambda_1^{-\epsilon}}\lambda^{-\epsilon}\text{ for all }\lambda\geq \lambda_1.
\end{equation}
Hence we have
\begin{equation}\label{show.proper.0}
 \int_{\lambda_1}^\infty\frac{\pi(\lambda)}{\lambda}\rd \lambda
\leq \frac{\pi(\lambda_1)}{\lambda_1^{-\epsilon}}\int_{\lambda_1}^\infty \frac{\rd\lambda }{\lambda^{1+\epsilon}}
= \frac{\pi(\lambda_1)}{\epsilon}<\infty.
\end{equation}

[Parts \ref{lem:pi:5} and \ref{lem:pi:6}] 
The proof is omitted since it is similar to that of Part \ref{lem:pi:4}.

[Part \ref{lem:pi:7}] 
  Under Assumptions \ref{AA3.1},
  by \eqref{eq:korene}, $\pi(\lambda)\to 0$ as $\lambda\to \infty$.
 Under Assumption \ref{AA3.2.1}, it is clear that $\pi(\lambda)$ is bounded.
  Under Assumption \ref{AA3.2.2}, there exist $\epsilon\in(0,1)$ and
  $\lambda_2>\exp(1)$ such that
\begin{equation}\label{eq:pi_log.0}
 \lambda\frac{\pi'(\lambda)}{\pi(\lambda)}\leq \frac{1-\epsilon}{\log \lambda}
\end{equation}
  for all $\lambda\geq \lambda_2$. 
As in \eqref{eq:korene.0} and \eqref{eq:korene}, we have
\begin{gather*}
 \int_{\lambda_2}^\lambda \frac{\pi'(s)}{\pi(s)}\rd s\leq 
(1-\epsilon)\int_{\lambda_2}^\lambda\frac{\rd s}{s\log s} \\
 \Leftrightarrow \log \frac{\pi(\lambda)}{\pi(\lambda_2)}
 \leq (1-\epsilon)\left\{\log\log \lambda-\log\log \lambda_2\right\}
\end{gather*}
  and hence
  \begin{equation}\label{eq:pi_log}
\pi(\lambda)\leq \pi(\lambda_2)\{\log \lambda\}^{1-\epsilon}\text{ for all }\lambda\geq \lambda_2,
  \end{equation}
  which completes the proof.
  
  [Part \ref{lem:pi:8}] 
  Under Assumption \ref{AA3.2.1}, there exists $\lambda_3>0$ such that
  $|\kappa(\lambda)|$ for $\lambda\geq \lambda_3$ is monotone decreasing. 
Then  $\pi(\lambda)$ for $\lambda\geq \lambda_3$ is expressed by
  \begin{align*}
   \pi(\lambda)=\pi(\lambda_3)\exp\left(-\int_{\lambda_3}^\lambda \frac{|\kappa(s)|}{s}\rd s\right)
  \end{align*}
  and
  \begin{align*}
 \int_{\lambda_3}^\infty
 \frac{\pi(\lambda)}{\lambda}\kappa^2(\lambda)\rd \lambda 
 &=\pi(\lambda_3) \int_{\lambda_3}^\infty \frac{\{\kappa(\lambda)\}^2}{\lambda}
 \exp\left(-\int_{\lambda_3}^\lambda \frac{|\kappa(s)|}{s}\rd s\right)\rd \lambda \\
 &\leq \pi(\lambda_3)|\kappa(\lambda_3)| \int_{\lambda_3}^\infty \frac{|\kappa(\lambda)|}{\lambda}
 \exp\left(-\int_{\lambda_3}^\lambda \frac{|\kappa(s)|}{s}\rd s\right)\rd \lambda \\
 &= \pi(\lambda_3)|\kappa(\lambda_3)|
\left[-\exp\left(-\int_{\lambda_3}^\lambda \frac{|\kappa(s)|}{s}\rd s\right)\right]_{\lambda_3}^\infty \\
 &\leq \pi(\lambda_3)|\kappa(\lambda_3)|<\infty.
  \end{align*}
  Under Assumption \ref{AA3.2.2}, by \eqref{eq:pi_log}, we have
  \begin{align*}
 \int_{\lambda_2}^\infty
 \frac{\pi(\lambda)}{\lambda}\kappa^2(\lambda)\rd \lambda 
\leq 
   \int_{\lambda_2}^\infty \frac{\pi(\lambda_2)}{\lambda \{\log \lambda\}^{1+\epsilon}}\rd \lambda 
=\frac{\pi(\lambda_2)}{\epsilon\{\log \lambda_2\}^{\epsilon}}<\infty,
  \end{align*}
which completes the proof.
 \end{proof}

\begin{remark}
 By Parts  \ref{lem:pi:2} and \ref{lem:pi:6} of Lemma \ref{lem:pi},
if $\lim_{\lambda\to\infty}\lambda\pi'(\lambda)/\pi(\lambda)<-1$,
the prior $\pi(\lambda)$ with Assumption \ref{AA2}
is proper and hence Part \ref{thm:proper.admissible.3} of Theorem \ref{thm:proper.admissible} can be applied. 
And this is why we assume $\lim_{\lambda\to\infty}\lambda\pi'(\lambda)/\pi(\lambda)\geq -1$
as the asymptotic behavior in Assumption \ref{AA3}.
\end{remark}

\subsection{The sequence $\pi_i$}
\label{sec:assumption_h}
The function $h_i(\lambda)$ in \eqref{eq:new_hir}
 satisfies the following.
 \begin{lemma}\label{lem:h_i}
  \begin{enumerate}
   \item \label{lem:h_i.1}
$h_i(\lambda)$ is increasing in $i$ for fixed $\lambda$, and decreasing in $\lambda$ for fixed $i$. Further
$\lim_{i\to\infty}h_i(\lambda)=1$ for fixed $\lambda\geq 0$.
   \item \label{lem:h_i.2}
	 For fixed $i$,
\begin{align*}
 \lim_{\lambda\to\infty} \{(\lambda+e+i) \log (\lambda+e+i) \log\log (\lambda+e+i)\} h_i(\lambda)=i.
\end{align*}
   \item \label{lem:h_i.4}
	 For $\lambda\geq 0$,
\begin{align*}
\sup_i |h'_i(\lambda)| \leq \frac{2}{(\lambda +e) \log (\lambda+e) \log\log (\lambda+e+1)}.
\end{align*}	 
\item \label{lem:h_i.4.5} $h_1(1)>1/8$.
\item \label{lem:h_i.4.6} $\sup_{i,\lambda} |h'_i(\lambda)|<5$.
   \item \label{lem:h_i.4.75} Under Assumption \ref{AA2} on $\pi$,
\begin{align*}
 \int_0^1 \pi_1(\lambda)\rd \lambda>0.
\end{align*}
   \item \label{lem:h_i.5}
	 Under Assumptions \ref{AA1},  \ref{AA2} and \ref{AA3} on $\pi$,
\begin{align*}
 \int_0^\infty \lambda \pi(\lambda)\sup_i\{h'_i(\lambda)\}^2   \rd  \lambda <\infty.
\end{align*}
   \item \label{lem:h_i.6}
	 Under Assumptions \ref{AA1}, \ref{AA2} and \ref{AA3} on $\pi$, 
\begin{align*}
 \int_0^\infty  \pi_i(\lambda)   \rd  \lambda <\infty,\text{ for fixed }i.
\end{align*}
  \end{enumerate}
 \end{lemma}
\begin{proof}\mbox{}
 [Part \ref{lem:h_i.1}] The part is straightforward given the form of $h_i(\lambda)$.

 [Part \ref{lem:h_i.2}] This follows from the expression,
 \begin{align*}
 h_i(\lambda)
  &=\frac{1}{\log\log(\lambda+e+i)} \log\frac{\log(\lambda+e+i)}{\log(\lambda+e)}\\
 &=\frac{\log(\{\lambda+e+i\}/\{\lambda+e\})}{\log(\lambda+e+i)\log\log(\lambda+e+i)}\zeta\left(\frac{\log(\{\lambda+e+i\}/\{\lambda+e\})}{\log(\lambda+e+i)}\right) \\
 &=\frac{i}{(\lambda+e+i)\log(\lambda+e+i)\log\log(\lambda+e+i)} \\
 &\quad\times \zeta\left(\frac{i}{\lambda+e+i}\right)\zeta\left(\frac{\log(\{\lambda+e+i\}/\{\lambda+e\})}{\log(\lambda+e+i)}\right),
 \end{align*}
where $\zeta(x)=-\log(1-x)/x$
 which satisfies $\lim_{x\to 0+}\zeta(x)=1$. 

[Part \ref{lem:h_i.4}] The derivative is 
 \begin{align*}
  h'_i(\lambda)&=-\frac{1}{(\lambda+e)\log(\lambda+e)\log\log(\lambda+e+i)} \\
  &\quad +\frac{\log\log(\lambda+e)}{(\lambda+e+i)\log(\lambda+e+i)\{\log\log(\lambda+e+i)\}^2}.
 \end{align*}
Hence we have
\begin{align*}
 |h'_i(\lambda)|&\leq\left|\frac{1}{(\lambda+e)\log(\lambda+e)\log\log(\lambda+e+i)}\right| \\ &\quad +
 \left|\frac{\log\log(\lambda+e)}{(\lambda+e+i)\log(\lambda+e+i)\{\log\log(\lambda+e+i)\}^2}\right| \\
 &\leq \frac{2}{(\lambda+e)\log(\lambda+e)\log\log(\lambda+e+1)}
 \end{align*} 
which does not depend on $i$.

 [Part \ref{lem:h_i.4.5}] At $\lambda=1$, $h_1(\lambda)$ is
 \begin{align*}
  h_1(1)=1-\frac{\log\log(1+e)}{\log\log(2+e)} =
  1-\frac{\int_e^{1+e} 1/(\lambda\log \lambda)\rd \lambda}{\int_e^{2+e} 1/(\lambda\log \lambda)\rd \lambda} 
  =\frac{\int_{1+e}^{2+e} 1/(\lambda\log \lambda)\rd \lambda}{\int_e^{2+e} 1/(\lambda\log \lambda)\rd \lambda},
\end{align*}
which is greater than
\begin{equation*}
\frac{1/\{(2+e)\log(2+e)\}}{2(1/e)} >\frac{1}{2}\frac{e}{2+e}\frac{1}{\log e^2} >\frac{1}{8}.
\end{equation*}
[Part \ref{lem:h_i.4.6}]
 The upper bound of $\sup_i|h'_i(\lambda)|$, derived in Part \ref{lem:h_i.4},
 is decreasing in $\lambda$ and hence
 \begin{align*}
  \sup_i|h'_i(\lambda)|\leq \sup_i|h'_i(\lambda)|\big|_{\lambda=0}=\frac{2}{e\log\log(e+1)}\leq \frac{1}{\log\log(e+1)}.
 \end{align*}
 Further we have
 \begin{gather*}
  \log\log(e+1)=\int_e^{e+1}\frac{\rd s}{s\log s} >\frac{\log(e+1)-\log(e)}{\log(e+1)}
=1-\frac{1}{\log(e+1)}, \\
\log(e+1)=\log(e)+\log\frac{e+1}{e}=1-\log\left(1-\frac{1}{e+1}\right)>1+\frac{1}{e+1} ,
 \end{gather*}
and hence $ \sup_{\lambda,i}|h'_i(\lambda)|\leq e+2<5$.

 [Part \ref{lem:h_i.4.75}]
 By Parts \ref{lem:h_i.1} and \ref{lem:h_i.4.5},
 $h^2_1(\lambda)\geq 1/64$ for $\lambda\in[0,1]$.
 By Assumption \ref{AA2} on $\pi$, there exists $\lambda_1\in(0,1)$
 such that $ \pi(\lambda) \geq \lambda^{\alpha}\{\nu(0)/2\}$ for $\lambda\in[0,\lambda_1]$.
 Then
 \begin{align*}
  \int_0^1 \pi(\lambda)h^2_1(\lambda)\rd \lambda \geq 
  \frac{\nu(0)}{2}\frac{1}{64}\int_0^{\lambda_1} \lambda^{\alpha}\rd \lambda 
=  \frac{\nu(0)\lambda_1^{\alpha+1}}{128(\alpha+1)} >0.
 \end{align*}
 [Part \ref{lem:h_i.5}] 
 As in Part \ref{lem:pi:7} of Lemma \ref{lem:pi}, there exist $\epsilon\in(0,1)$ and $\lambda_2>\exp(1)$ such that
 \begin{equation}\label{pi:asymptotic}
\pi(\lambda)\leq \pi(\lambda_2)\{\log \lambda\}^{1-\epsilon}\text{ for all }\lambda\geq \lambda_2.
 \end{equation}
 Then, by Part \ref{lem:h_i.4.6} and \eqref{pi:asymptotic}, we have
 \begin{align*}
&\int_0^\infty \lambda\pi(\lambda)\sup_i \{h'_i(\lambda)\}^2 \rd \lambda   \\
  &\leq
  25
  \int_0^{\lambda_2} \lambda  \pi(\lambda)  \rd  \lambda 
  +  \int_{\lambda_2}^\infty \lambda \pi(\lambda)\sup_i\{h'_i(\lambda)\}^2   \rd  \lambda \\
  &\leq 25  \int_0^{\lambda_2} \lambda  \pi(\lambda)  \rd  \lambda
  +  \pi(\lambda_2)\int_{\lambda_2}^\infty
 \frac{4 (\lambda+e)\log (\lambda+e)\rd  \lambda}{\left\{(\lambda+e)\log (\lambda+e)\log\log (\lambda+e)\right\}^2}    \\ 
  &=
  25
  \int_0^{\lambda_2} \lambda  \pi(\lambda)  \rd  \lambda 
  +  \frac{4\pi(\lambda_2)}{\log\log(\lambda_2+e)},
 \end{align*}
 where
$\int_0^{\lambda_2} \lambda  \pi(\lambda)\rd  \lambda $ in the first term of the right-hand side 
is bounded  by Part \ref{lem:pi:2} of Lemma \ref{lem:pi}.

 [Part \ref{lem:h_i.6}] The proof is omitted since it is similar to the proof of Part \ref{lem:h_i.5}.
\end{proof}

\subsection{Assumption on $f$}
 
\begin{lemma}\label{lem:f}
 Let Assumptions \ref{FF1}--\ref{FF3} hold.
\begin{enumerate}
 \item\label{lem:f.1} Also assume
 \begin{equation}\label{eq:lem:f.1}
 \limsup_{t\to\infty}\, t\frac{f'(t)}{f(t)}<-\frac{p+n}{2}-2-j
 \end{equation}
      for $j\geq 0$ (hence $j=0$ for Assumption \ref{FF3.1} and $j=1$ for Assumption \ref{FF3.2}).
\begin{enumerate}[label= 1.\Alph*]
 \item\label{lem:f.1.1} Then there exist $\epsilon\in(0,1)$ and $t_*>1$ such that
\begin{equation}\label{fF.asymptotic}
 \begin{split}
 f(t)&\leq \frac{f(t_*)}{t_*^{-(p+n)/2-2-j-\epsilon}}t^{-(p+n)/2-2-j-\epsilon}, \\
F(t)&\leq \frac{tf(t)}{(p+n)+2+2j+2\epsilon},
\end{split}
\end{equation}
for all $t\geq t_*$,  where
 \begin{align*}
  F(t)=\frac{1}{2}\int_t^\infty f(s)\rd s.
 \end{align*}
\item \label{lem:f.1.2}
 \begin{align*}
  \int_0^\infty t^{(p+n)/2-1+j}\left\{\frac{F(t)}{f(t)}\right\}^2f(t) \rd t<\infty.
 \end{align*}
\end{enumerate}
\item\label{lem:f.2} Assume Assumption \ref{FF3.2}. Also assume $p\geq 3$. Let 
\begin{align*}
\tilde{\mathcal{F}}(t)&=t^{1/2}F(t)/f(t), \\
 f_\star(t)&=\int_0^\infty \eta^{n/2-1}\tilde{\mathcal{F}}^2(t+\eta)f(t+\eta)
 \rd \eta.
\end{align*}
      Then there exists $\mathcal{Q}_f>0$ such that
      \begin{align}
       \int_{\mathbb{R}^p}\frac{1}{\|y\|^2}f_\star(\|y-\mu\|^2)\rd y \leq \mathcal{Q}_f\min(1,1/\|\mu\|^{2}).
      \end{align}
\end{enumerate}
\end{lemma}
\begin{proof}\mbox{}
[Part \ref{lem:f.1.1}] By \eqref{eq:lem:f.1}, there exist $t_*>1$ and $\epsilon\in(0,1)$ such that
\begin{equation}\label{eq:ffff}
 t\frac{f'(t)}{f(t)}\leq -\frac{p+n}{2}-2-\epsilon-j
\end{equation}
for all $t\geq t_*$.
Then, by \eqref{eq:ffff}, we have
\begin{align}
& \int_{t_*}^t \frac{f'(s)}{f(s)} \rd s
 \leq \left(-\frac{p+n}{2}-2-j-\epsilon\right) \int_{t_*}^t\frac{ \rd s}{s} \quad\text{ for }t\geq t_*, \notag \\
&\Leftrightarrow \log \frac{f(t)}{f(t_*)} \leq \left(-\frac{p+n}{2}-2-j-\epsilon\right)\log \frac{t}{t_*} \quad\text{ for }t\geq t_*, \notag \\
& \Leftrightarrow  f(t)\leq \frac{f(t_*)}{t_*^{-(p+n)/2-2-j-\epsilon}}t^{-(p+n)/2-2-j-\epsilon}
 \quad\text{ for }t\geq t_*. \label{proof.f.1.2}
\end{align}
Further, by \eqref{eq:ffff}, we have
\begin{align*}
 t f'(t) \leq -\left(\frac{p+n}{2}+2+j+\epsilon\right)f(t),
\end{align*}
for all $t\geq t_*$, and hence
 \begin{equation}\label{eq:lem:f.1.1}
 \int_t^\infty s f'(s)\rd s \leq -\left(\frac{p+n}{2}+2+j+\epsilon\right)\int_t^\infty f(s)\rd s.  
 \end{equation}
 By an integration by parts, the left-hand side
 is rewritten as
 \begin{equation*}
  \int_t^\infty s f'(s)\rd s
   =[sf(s)]_t^\infty - \int_t^\infty  f(s)\rd s   = -tf(t)- 2F(t),
 \end{equation*}
 where the second equality follows from $[sf(s)]_t^\infty=-tf(t)$  by \eqref{proof.f.1.2}.
Then the inequality \eqref{eq:lem:f.1.1} is equivalent to
\begin{equation}\label{eq:F.-1}
 \begin{split}
 -tf(t)- 2F(t) &\leq -2\left(\frac{p+n}{2}+2+j+\epsilon\right)F(t), \\
\Leftrightarrow\quad  \frac{F(t)}{f(t)}&\leq\frac{t}{(p+n)+2+2j+2\epsilon},
 \end{split}
\end{equation} 
for all $t\geq t_*$. Hence Part \ref{lem:f.1.1} follows from \eqref{proof.f.1.2} and \eqref{eq:F.-1}.

 \medskip

[Part \ref{lem:f.1.2}]  By Assumption \ref{FF1}, we have
\begin{equation}\label{eq:F.0}
 \int_0^1 f(s) \rd s<\infty.
\end{equation}
 Also the integrability given by \eqref{assmp.f.1},
 \begin{align*}
 \int_1^\infty s^{(p+n)/2-1}f(s) \rd s<\infty,
 \end{align*}
implies
\begin{equation}\label{eq:F.1}
 \int_1^\infty f(s) \rd s<\infty.
\end{equation}
By \eqref{eq:F.0} and \eqref{eq:F.1}, we have
\begin{equation}\label{eq:F.2}
F(0)= \frac{1}{2}\int_0^\infty f(s) \rd s<\infty.
\end{equation}
Note $0<f(0)<\infty$ by Assumption \ref{FF1}. Also by \eqref{eq:F.-1} and \eqref{eq:F.2},
it follows that there exists $C_f>0$ such that
\begin{equation}\label{eq:F.3}
 \frac{F(t)}{f(t)}\leq C_f\max(t,t_*),\quad \forall t\geq 0.
\end{equation}
 By \eqref{eq:F.3}, for $t\in[0,1]$, we have
\begin{equation}\label{proof.f.1}
t^{j}\left\{\frac{F(t)}{f(t)}\right\}^2f(t) \leq C_f^2t_*^2\max_{t\in[0,1]}f(t) 
\end{equation}
 and hence
\begin{equation}\label{proof.f.1.1}
 \int_0^1 t^{(p+n)/2-1+j}\left\{\frac{F(t)}{f(t)}\right\}^2f(t) \rd t  
\leq \frac{2C_f^2t_*^2}{p+n}\max_{t\in[0,1]}f(t) <\infty.
\end{equation}
By \eqref{proof.f.1.2} and \eqref{eq:F.3},  we have
\begin{equation}\label{proof.f.2}
 t^{j}\left\{\frac{F(t)}{f(t)}\right\}^2f(t) \leq \frac{f(t_*)C_f^2}{t_*^{-(p+n)/2-2-j-\epsilon}}
  t^{-(p+n)/2-\epsilon}
\end{equation}
for $t\geq t_*$ and hence
\begin{equation}\label{proof.f.1.1.3}
 \begin{split}
 \int_{t_*}^\infty t^{(p+n)/2-1+j}\left\{\frac{F(t)}{f(t)}\right\}^2f(t) \rd t 
&\leq \frac{f(t_*)C_f^2}{t_*^{-(p+n)/2-2-j-\epsilon}}\int_{t_*}^\infty t^{-1-\epsilon}\rd t \\
  &= \frac{f(t_*)C_f^2}{\epsilon t_*^{-(p+n)/2-2-j}} <\infty.  
 \end{split}
\end{equation} 
Combining \eqref{proof.f.1.1} and \eqref{proof.f.1.1.3}, completes the proof of Part \ref{lem:f.1.2}.

 \medskip
 
 [Part \ref{lem:f.2}]
Note, by Part \ref{lem:f.1} of this lemma with $j=1$,
\begin{equation}\label{eq:lem:f.2.1}
 \begin{split}
  \int_0^\infty t^{(p+n)/2-1+1}\left\{\frac{F(t)}{f(t)}\right\}^2f(t) \rd t 
  & = \int_0^\infty t^{(p+n)/2-1}\tilde{\mathcal{F}}^2(t)f(t) \rd t \\
  &<\infty.
 \end{split}
\end{equation} 
To prove Part \ref{lem:f.2}, it suffices to show that, for $\|\mu\|=0$,
      \begin{align}\label{proof.f.2.around.zero}
       \int_{\mathbb{R}^p}\frac{1}{\|y\|^2}f_\star(\|y\|^2)\rd y <\infty
      \end{align}
and also that there exist $a>0$ and $b>0$ such that 
      \begin{align}\label{proof.f.2.around.infty}
      \|\mu\|^2 \int_{\mathbb{R}^p}\frac{1}{\|y\|^2}f_\star(\|y-\mu\|^2)\rd y <b 
      \end{align}
 for all $\|\mu\|^2\geq a$.
 
[Bound in \eqref{proof.f.2.around.zero}] 
Note $f_\star(0)$ is decomposed as
\begin{equation}\label{eq:f_star_dayo}
 \begin{split}
 f_\star(0)&=\int_0^\infty \eta^{n/2-1}\tilde{\mathcal{F}}^2(\eta)f(\eta)\rd \eta \\
  &=\int_0^1\eta^{n/2-1}\tilde{\mathcal{F}}^2(\eta)f(\eta)\rd \eta+
  \int_1^{t_*}\eta^{n/2-1}\tilde{\mathcal{F}}^2(\eta)f(\eta)\rd \eta \\
&\qquad + \int_{t_*}^\infty\eta^{n/2-1}\tilde{\mathcal{F}}^2(\eta)f(\eta)\rd \eta ,
 \end{split}
\end{equation} 
where $t_*$ is from \eqref{eq:ffff}.
The first and third terms both are integrable since,
 by \eqref{proof.f.1},
\begin{equation}\label{eq:proof.f.2.0.1}
 \begin{split}
  \int_0^1\eta^{n/2-1}\tilde{\mathcal{F}}^2(\eta)f(\eta)\rd \eta 
&\leq C_f^2t_*^2\max_{\eta\in[0,1]}f(\eta)
 \int_0^1\eta^{n/2-1}\rd\eta \\
& =C_f^2t_*^2\frac{2}{n}\max_{\eta\in[0,1]}f(\eta),
 \end{split}
\end{equation}
 and by \eqref{proof.f.2},
\begin{equation}\label{eq:proof.f.2.t_*infty}
 \begin{split}
  \int_{t_*}^\infty \eta^{n/2-1}\tilde{\mathcal{F}}^2(\eta)f(\eta)\rd \eta 
  &\leq
\frac{f(t_*)C_f^2}{t_*^{-(p+n)/2-3-\epsilon}}
\int_{t_*}^\infty  \eta^{n/2-1-(p+n)/2-\epsilon} \rd \eta \\
  & =
\frac{f(t_*)C_f^2}{t_*^{-(p+n)/2-3-\epsilon}}\frac{t_*^{-p/2-\epsilon}}{p/2+\epsilon}.
 \end{split}
\end{equation}
 By \eqref{eq:f_star_dayo}, \eqref{eq:proof.f.2.0.1} and \eqref{eq:proof.f.2.t_*infty}, we have
$f_\star(0)<\infty$. Then, by continuity of $f_\star$, it follows that
 \begin{equation}\label{eq:f_star_0_1}
\sup_{t\in[0,1]}f_\star(t)<\infty.
 \end{equation}
  Further the integrability of $\int_{\mathbb{R}^p}f_\star(\|y\|^2)\rd y $ follows since
\begin{equation}\label{eq:int_f_star}
 \begin{split}
     \int_{\mathbb{R}^p}f_\star(\|y\|^2)\rd y 
& = \int_{\mathbb{R}^p}\int_0^\infty \eta^{n/2-1}\tilde{\mathcal{F}}^2(\|y\|^2+\eta)f(\|y\|^2+\eta)\rd \eta \rd y\\
 &=\frac{1}{c_n}\int_{\mathbb{R}^{p+n}}\tilde{\mathcal{F}}^2(\|q\|^2)f(\|q\|^2)\rd q \\
   &=\frac{c_{p+n}}{c_n}\int_0^\infty t^{(p+n)/2-1}\tilde{\mathcal{F}}^2(t)f(t)\rd t \\
   &<\infty \ (\text{by \eqref{eq:lem:f.2.1}}).
  \end{split}
\end{equation}  
 Then, by \eqref{eq:f_star_0_1} and \eqref{eq:int_f_star}, we have
      \begin{align*}
       \int_{\mathbb{R}^p}\frac{f_\star(\|y\|^2)}{\|y\|^2}\rd y 
&       \leq \sup_{\|y\|\leq 1}f_\star(\|y\|^2) \int_{\|y\|\leq 1}\frac{\rd y}{\|y\|^2}
+\int_{\|y\|\geq 1}f_\star(\|y\|^2)\rd y \\
&\leq \frac{2c_p }{p-2}\sup_{\|y\|\leq 1}f_\star(\|y\|^2)+
       \int_{\mathbb{R}^p}f_\star(\|y\|^2)\rd y \\
       &<\infty.
      \end{align*}
 Hence the bound in \eqref{proof.f.2.around.zero} is established.
 
 [Bound in \eqref{proof.f.2.around.infty}] Let $ \|\mu\|^2>2t_*$
where $t_*$ is from \eqref{eq:ffff}.
Under the decomposition  of the integral region,
\begin{align*}
 \mathbb{R}^p
 &= \left\{y:\|y-\mu\|^2\leq \|\mu\|^2/2\right\} \\ &\qquad \cup
 \left\{y:\|y-\mu\|^2\geq \|\mu\|^2/2\text{ and }0\leq \|y\|^2\leq \|\mu\|^2\right\}
 \\ &\qquad \cup
 \left\{y:\|y-\mu\|^2\geq \|\mu\|^2/2\text{ and } \|y\|^2\geq \|\mu\|^2\right\} \\
&=\mathcal{R}_1\cup\mathcal{R}_2\cup\mathcal{R}_3,
\end{align*}
 we have
\begin{align*}
&\|\mu\|^2 \int_{\mathbb{R}^p}\frac{f_\star(\|y-\mu\|^2)}{\|y\|^2}\rd y  \\
 &= \|\mu\|^2 \left(\int_{\mathcal{R}_1}+\int_{\mathcal{R}_2}+\int_{\mathcal{R}_3}\right)
 \frac{f_\star(\|y-\mu\|^2)}{\|y\|^2}\rd y .
\end{align*}
For the region $\mathcal{R}_1$, $\|y-\mu\|^2\leq \|\mu\|^2/2$ implies $\|y\|^2\geq \|\mu\|^2/2$
 and hence
\begin{align*}
 \|\mu\|^2 \int_{\mathcal{R}_1}\frac{f_\star(\|y-\mu\|^2)}{\|y\|^2}\rd y
&\leq 2\int_{\mathcal{R}_1}f_\star(\|y-\mu\|^2)\rd y \\
&\leq 2\int_{\mathbb{R}^p}f_\star(\|y-\mu\|^2)\rd y, 
\end{align*}
which is bounded by \eqref{eq:int_f_star}. 
 Similarly, for $\mathcal{R}_1$, since $\|y\|^2\geq \|\mu\|^2$, we have
 \begin{align*}
   \|\mu\|^2 \int_{\mathcal{R}_3}\frac{f_\star(\|y-\mu\|^2)}{\|y\|^2}\rd y
\leq \int_{\mathbb{R}^p}f_\star(\|y-\mu\|^2)\rd y  <\infty.
 \end{align*}

 For the region
\begin{align*}
\mathcal{R}_2= \left\{y:\|y-\mu\|^2\geq \|\mu\|^2/2\text{ and }0\leq \|y\|^2\leq \|\mu\|^2\right\},
\end{align*}
 we have
\begin{align*}
 \mathcal{R}_2\subset\left\{y:\|y-\mu\|^2\geq \|\mu\|^2/2\right\}, \quad
\mathcal{R}_2\subset\left\{y:0\leq \|y\|^2\leq \|\mu\|^2\right\}.
\end{align*}
 Hence
\begin{equation}\label{eq:Rtwo.1}
 \begin{split}
&   \|\mu\|^2 \int_{\mathcal{R}_2}\frac{f_\star(\|y-\mu\|^2)}{\|y\|^2}\rd y \\
&\leq \sup_{y:\|y-\mu\|^2\geq \|\mu\|^2/2}f_\star(\|y-\mu\|^2)
  \int_{y:0\leq \|y\|^2\leq \|\mu\|^2}\frac{\|\mu\|^2}{\|y\|^2}\rd y, 
\end{split}
\end{equation}
 where
\begin{equation}\label{eq:Rtwo.2}
 \int_{y:0\leq \|y\|^2\leq \|\mu\|^2}\frac{\|\mu\|^2}{\|y\|^2}\rd y 
 = \|\mu\|^2 c_p \int_0^{\|\mu\|^2} r^{p/2-2}\rd r=\frac{2c_p}{p-2}\|\mu\|^p.
\end{equation} 
Recall the assumption $\|\mu\|^2>2t_*$ and
 hence note
 \begin{equation}\label{eq:region_R2}
  \|y-\mu\|^2\geq \|\mu\|^2/2 > t_*.
 \end{equation}
By \eqref{proof.f.2}, the integrand of $f_\star$, for $t\geq t_*$, is bounded as
\begin{align*}
 \tilde{\mathcal{F}}^2(t)f(t)
\leq \frac{f(t_*)C_f^2}{t_*^{-(p+n)/2-3-\epsilon}} t^{-(p+n)/2-\epsilon}
\end{align*}
 and hence $f_\star(t)$ for $t\geq t_*$ is bounded as
\begin{equation}\label{eq:int_f_star_R2} 
 \begin{split}
f_\star(t)&= \int_0^\infty \eta^{n/2-1}\tilde{\mathcal{F}}^2(t+\eta)f(t+\eta)\rd \eta \\
&\leq \frac{f(t_*)C_f^2}{t_*^{-(p+n)/2-3-\epsilon}}\int_0^\infty \eta^{n/2-1}(t+\eta)^{-(p+n)/2-\epsilon} \rd \eta \\
 &= \tilde{C}_f t^{-p/2-\epsilon},
\end{split}
\end{equation}
 where
 \begin{align*}
  \tilde{C}_f=f(t_*)C_f^2 t_*^{(p+n)/2+3+\epsilon}B(p/2+\epsilon,n/2).
 \end{align*}
 Then,  for any $y\in \{y:\|y-\mu\|^2\geq \|\mu\|^2/2\}$ with $\|\mu\|^2>2t_*$,
\begin{equation}\label{eq:Rtwo.3}
 f_\star(\|y-\mu\|^2)\leq \tilde{C}_f \{\|y-\mu\|^2\}^{-p/2-\epsilon}
 \leq \frac{\tilde{C}_f}{2^{p/2+\epsilon}}\|\mu\|^{-p-2\epsilon},
\end{equation}
where the first and second inequalities follow from \eqref{eq:int_f_star_R2} and
\eqref{eq:region_R2}, respectively.
By \eqref{eq:Rtwo.1}, \eqref{eq:Rtwo.2}, and \eqref{eq:Rtwo.3},
 \begin{align*}
  \|\mu\|^2 \int_{\mathcal{R}_2}\frac{f_\star(\|y-\mu\|^2)}{\|y\|^2}\rd y 
  \leq \frac{1}{\|\mu\|^{2\epsilon}}
  \left\{\tilde{C}_f 2^{p/2+\epsilon}\frac{2c_p}{p-2} \right\}
 \end{align*}
which is bounded under the assumption $\|\mu\|^2>2t_*$.
\end{proof}

\section{Preliminary results for completing Proof of Theorem \ref{thm:mainmain}}
\label{sec:ap.pre.main}
Note that the first three parts \ref{BL.-1}, \ref{BL.0} and \ref{BL.1}
of \citeapos{Blyth-1951} conditions needed to prove Theorem \ref{thm:mainmain}
follow from 
Parts \ref{lem:h_i.1}, \ref{lem:h_i.6} and \ref{lem:h_i.4.75} of Lemma \ref{lem:h_i},
respectively.
In this appendix we provide an alternative expression $\oldiff(z;\delta_\pi,\delta_{\pi i};\pi_i)$ in \ref{BL.2}.
The proof of Theorem \ref{thm:mainmain}'s two cases, I and II is completed in the
two succeeding sections \ref{sec:caseI} and \ref{sec:caseII} respectively
using this re-expression.

Recall
as in \eqref{eq:ndiff} and \eqref{eq:oldiff},
\begin{equation}\label{eq:ndiff.oldiff.ap}
 \begin{split}
 \ndiff(\delta_\pi,\delta_{\pi i};\pi_i) 
&=c_n\int_{\mathbb{R}^p}\oldiff(z;\delta_\pi,\delta_{\pi i};\pi_i) \rd z, \\
 \oldiff(z;\delta_\pi,\delta_{\pi i};\pi_i)
  &=\{\psi_{\pi }(\|z\|^2)-\psi_{\pi i}(\|z\|^2)\}^2\|z\|^2M_1(z,\pi_i), \\
\end{split}
\end{equation}
with
\begin{align*}
 \psi_\pi(z)&=1- \frac{z^\T M_2(z,\pi)}{\|z\|^2 M_1(z,\pi)}=\frac{z^\T z M_1(z,\pi)- z^\T M_2(z,\pi)}{\|z\|^2 M_1(z,\pi)}, \\
 M_1(z,\pi)&=
 \iint \eta^{(2p+n)/2}f(\eta\{\|z-\theta\|^2+1\})
\bar{\pi}(\eta\|\theta\|^2) \rd \theta   \rd \eta,  \\
 M_2(z,\pi)&=
\iint \theta
  \eta^{(2p+n)/2} f(\eta\{\|z-\theta\|^2+1\}) \bar{\pi}(\eta\|\theta\|^2) \rd \theta  \rd \eta.
\end{align*}
The numerator of $\psi_\pi(z)$ is rewritten as
\begin{equation}\label{eq:numerator.another.exp}
 \begin{split}
&z^\T z M_1(z,\pi)- z^\T M_2(z,\pi) \\
&= z^\T\iint \eta (z-\theta)
  \eta^{(2p+n)/2-1} f(\eta\{\|z-\theta\|^2+1\}) 
\bar{\pi}(\eta\|\theta\|^2) \rd \theta  \rd \eta \\
 &= z^\T\iint 
 \eta^{(2p+n)/2-1} \nabla_\theta F(\eta\{\|z-\theta\|^2+1\})
 \bar{\pi}(\eta\|\theta\|^2) \rd \theta   \rd \eta \\
 &= -z^\T\iint 
  \eta^{(2p+n)/2-1} F(\eta\{\|z-\theta\|^2+1\}) \nabla_\theta\bar{\pi}(\eta\|\theta\|^2) \rd \theta   \rd \eta,
\end{split}
\end{equation}
where the last equality follows from an integration by parts. 
To justify this integration by parts, note that, for fixed $\theta_i$, the $i$-th component of $\theta$, we have
\begin{equation*}
 \lim_{\theta_i\to\pm\infty}F(\eta\{\|z-\theta\|^2+1\}) \bar{\pi}(\eta\|\theta\|^2)=0
\end{equation*}
for any fixed $\eta$, $z$, $\theta_1,\dots,\theta_{i-1},\theta_{i+1},\dots,\theta_p$,
since the asymptotic behavior of $\bar{\pi}$ and $F$ are 
given by
\begin{align*}
 \bar{\pi}(\lambda)=c_p^{-1}\lambda^{1-p/2}\pi(\lambda)=o(\lambda^{1-p/2}\log\lambda)
\text{ and }F(t)=o(t^{-(p+n)/2-1}),
\end{align*}
as in Part \ref{lem:pi:7} of Lemma \ref{lem:pi} and
Part \ref{lem:f.1.1} of Lemma \ref{lem:f}, respectively.
Thus the last equality of \eqref{eq:numerator.another.exp} follows.

Therefore $\oldiff(z;\delta_\pi,\delta_{\pi i};\pi_i)$ is re-expressed as
\begin{equation}\label{diff.general}
\begin{split}
& \oldiff(z;\delta_\pi,\delta_{\pi i};\pi_i) \\
&=\left\|\frac{\iint 
  \eta^{(2p+n)/2-1} F(\eta\{\|z-\theta\|^2+1\}) 
 \nabla_\theta\bar{\pi}(\eta\|\theta\|^2) \rd \theta   \rd \eta}
 {\iint 
  \eta^{(2p+n)/2} f(\eta\{\|z-\theta\|^2+1\}) 
\bar{\pi}(\eta\|\theta\|^2) \rd \theta   \rd \eta}\right. \\
 &\qquad \left.
- \frac{\iint 
  \eta^{(2p+n)/2-1} F(\eta\{\|z-\theta\|^2+1\}) 
 \nabla_\theta\bar{\pi}_i(\eta\|\theta\|^2) \rd \theta   \rd \eta}
 {\iint 
  \eta^{(2p+n)/2} f(\eta\{\|z-\theta\|^2+1\}) 
\bar{\pi}_{i}(\eta\|\theta\|^2) \rd \theta   \rd \eta}\right\|^2 \\
 &\qquad\qquad\times \iint\eta^{(2p+n)/2} f(\eta\{\|z-\theta\|^2+1\}) 
\bar{\pi}_{i}(\eta\|\theta\|^2) \rd \theta   \rd \eta.
\end{split}
\end{equation}
The proof of Theorem \ref{thm:mainmain}, Cases I and II, will be completed in Sections
\ref{sec:caseI} and \ref{sec:caseII}
by showing $\ndiff(\delta_\pi,\delta_{\pi i};\pi_i) \to 0$ as $i\to\infty$.

 \section{Proof for Case I}
\label{sec:caseI}
By \eqref{diff.general} and the decomposition
\begin{align*}
 \nabla_\theta\bar{\pi}_i(\eta\|\theta\|^2)
 &= \nabla_\theta\left\{\bar{\pi}(\eta\|\theta\|^2)h_i^2(\eta\|\theta\|^2)\right\}\\
 &=\{\nabla_\theta\bar{\pi}(\eta\|\theta\|^2)\}h_i^2(\eta\|\theta\|^2)
+\bar{\pi}(\eta\|\theta\|^2)\{\nabla_\theta h_i^2(\eta\|\theta\|^2)\},
\end{align*}
 we have
\begin{align*}
& \oldiff(z;\delta_\pi,\delta_{\pi i};\pi_i) \\ 
&=c_n\left\|\frac{\iint 
  \eta^{(2p+n)/2-1} F(\circ) 
 \nabla_\theta\bar{\pi}(\bullet) \rd \theta   \rd \eta}
 {\iint 
  \eta^{(2p+n)/2} f(\circ) 
\bar{\pi}(\bullet) \rd \theta   \rd \eta}
- \frac{\iint 
  \eta^{(2p+n)/2-1} F(\circ) 
\nabla_\theta\bar{\pi}(\bullet)h_i^2(\bullet) \rd \theta   \rd \eta}
 {\iint 
  \eta^{(2p+n)/2} f(\circ) 
\bar{\pi}_i(\bullet) \rd \theta   \rd \eta} \right. \\
 &\qquad \left.
- \frac{\iint 
  \eta^{(2p+n)/2-1} F(\circ) 
\bar{\pi}(\bullet)\nabla_\theta h_i^2(\bullet) \rd \theta   \rd \eta}
 {\iint 
  \eta^{(2p+n)/2} f(\circ) 
\bar{\pi}_i(\bullet) \rd \theta   \rd \eta}\right\|^2 
 \iint\eta^{(2p+n)/2} f(\circ) 
\bar{\pi}_i(\bullet) \rd \theta   \rd \eta,
\end{align*}
where, for notational convenience and to control the size of expressions,
\begin{align*}
 \bullet=\eta\|\theta\|^2, \quad \circ=\eta(\|z-\theta\|^2+1).
\end{align*}
Further, by the triangle inequality and the fact $h_i^2\leq 1$, we have
\begin{align*}
\oldiff(z;\delta_\pi,\delta_{\pi i};\pi_i)
 \leq 2c_n(\Delta_{1i}+\Delta_{2i}),
\end{align*}
where
\begin{align}
 \Delta_{1i}
& =
\frac{\left\| \iint 
  \eta^{(2p+n)/2-1} F(\circ) 
\bar{\pi}(\bullet)\nabla_\theta h_i^2(\bullet) \rd \theta   \rd \eta \right\|^2 }
 {\iint \eta^{(2p+n)/2} f(\circ) \bar{\pi}_i(\bullet) \rd \theta   \rd \eta}, \label{eq:deldel1}\\
\Delta_{2i}
 &=
 \frac{\left\|\iint   \eta^{(2p+n)/2-1} F(\circ) \nabla_\theta\bar{\pi}(\bullet) \rd \theta   \rd \eta \right\|^2}
 {\iint \eta^{(2p+n)/2} f(\circ) \bar{\pi}(\bullet) \rd \theta   \rd \eta} \notag\\ 
 &\quad +
\frac{\left\|\iint  \eta^{(2p+n)/2-1} F(\circ) 
\nabla_\theta\bar{\pi}(\bullet)h_i^2(\bullet) \rd \theta   \rd \eta \right\|^2}
 {\iint \eta^{(2p+n)/2} f(\circ) \bar{\pi}_i(\bullet) \rd \theta   \rd \eta}. \label{eq:deldel3}
\end{align}
The proof of Case I will be completed by showing that each $\Delta_{ji}$ for $j=1,2$ is
bounded by an integrable function. The theorem then follows by
the dominated convergence theorem since
$\lim_{i\to\infty}\ndiff(\delta_\pi,\delta_{\pi i};\pi_i) =0$ since $h_i^2\to 1$
and $\delta_{\pi_i}\to \delta_\pi$ in the expression of \eqref{eq:ndiff.oldiff.ap}.

 \subsection{$\Delta_{1i}$}
\label{sec:del1}
Note
\begin{align*}
\nabla_\theta h_i^2(\eta\|\theta\|^2)&=2 h_i(\eta\|\theta\|^2)\nabla_\theta h_i(\eta\|\theta\|^2), \\
 \|\nabla_\theta h_i(\eta\|\theta\|^2)\|^2
 &=4\eta^2\|\theta\|^2\{h'_i(\eta\|\theta\|^2)\}^2.
\end{align*}
Then, by Cauchy-Schwarz inequality,
\begin{align*}
\Delta_{1i}&=\frac{\|\iint 
  \eta^{(2p+n)/2-1}  F(\circ)
 \bar{\pi}(\bullet)\{2h_i(\bullet)\nabla_\theta h_i(\bullet)\} \rd \theta   \rd \eta\|^2}
 {\iint\eta^{(2p+n)/2} f(\circ) \bar{\pi}_i(\bullet) \rd \theta   \rd \eta} \\
&\leq 4\iint 
  \eta^{(2p+n)/2-2} \mathcal{F}^2(\circ) f(\circ) 
\bar{\pi}(\bullet)\|\nabla_\theta h_i(\bullet)\|^2 \rd \theta   \rd \eta \\
&=16\iint 
  \eta^{(2p+n)/2-2} \mathcal{F}^2(\circ) f(\circ) 
\bar{\pi}(\bullet)\eta^2 \|\theta\|^2 \{h'_i(\bullet)\}^2  \rd \theta   \rd \eta,
\end{align*}
where $\mathcal{F}(t)=F(t)/f(t)$.
Then
\begin{align*}
\frac{\int_{\mathbb{R}^p}\sup_i\Delta_{1i} \rd z}{16} 
 &\leq
 \iiint 
 \eta^{(2p+n)/2-1} \mathcal{F}^2(\eta\{\|z-\theta\|^2+1\})
 f(\eta\{\|z-\theta\|^2+1\}) \\
 &\qquad \times\eta\|\theta\|^2\bar{\pi}(\eta\|\theta\|^2)\sup_i\{h'_i(\eta\|\theta\|^2)\}^2
 \rd \theta   \rd \eta  \rd  z \\
 &
= \iiint  \eta^{(p+n)/2-1} \mathcal{F}^2(\eta\{\|z\|^2+1\})
 f(\eta\{\|z\|^2+1\}) \\
 &\qquad \times\|\mu\|^2\bar{\pi}(\|\mu\|^2)\sup_i\{h'_i(\|\mu\|^2)\}^2
 \rd \mu   \rd \eta  \rd  z \\
 &
\leq A_1c_p  \iint w^{p/2-1}
 \eta^{(p+n)/2-1} \mathcal{F}^2(\eta\{w+1\}) f(\eta\{w+1\}) \rd  w  \rd \eta \\
 &= A_1c_p \int_0^\infty \frac{w^{p/2-1} \rd w}{(1+w)^{(p+n)/2}}
 \int_0^\infty t^{(p+n)/2-1} \mathcal{F}^2(t) f(t)  \rd t \\
 &\leq A_1A_2 c_p  B(p/2,n/2), 
\end{align*}
where
\begin{align}
 A_1&=\int_{\mathbb{R}^p} \|\mu\|^2\bar{\pi}(\|\mu\|^2)\sup_i \{h'_i(\|\mu\|^2)\}^2 \rd \mu \notag\\
 &=\int_0^\infty\lambda \pi(\lambda)\sup_i \{h'_i(\lambda)\}^2 \rd \lambda \label{proof.A_1} \\
 A_2&=\int_0^\infty t^{(p+n)/2-1} \mathcal{F}^2(t) f(t)  \rd t    \label{proof.A_2}
\end{align}
are both bounded as shown in Part \ref{lem:h_i.5} of Lemma \ref{lem:h_i} and in
Part \ref{lem:f.1} of Lemma \ref{lem:f}, respectively. 

\subsection{$\Delta_{2i}$}
\label{sec:del23}
We consider $\alpha>0$ and $-1/2<\alpha\leq 0$ separately in \ref{sec:subsub1}
and \ref{sec:subsub2}, respectively.
\subsubsection{$\Delta_{2i}$ under Assumption \ref{AA2} with $\alpha>0$}
\label{sec:subsub1}
By the Cauchy-Schwarz inequality, we have
\begin{equation}\label{eq:Del2.-1}
 \begin{split}
& \left\|\iint 
  \eta^{(2p+n)/2-1} F(\circ) 
\nabla_\theta\bar{\pi}(\bullet) \rd \theta   \rd \eta \right\|^2 \\
& \leq
 \iint   \eta^{(2p+n)/2-2} \mathcal{F}^2(\circ) f(\circ)
 \|\nabla_\theta\bar{\pi}(\bullet)/\bar{\pi}(\bullet)\|^2 \bar{\pi}(\bullet) \rd \theta   \rd \eta \\
 &\qquad \times
 \iint   \eta^{(2p+n)/2} f(\circ) \bar{\pi}(\bullet) \rd \theta   \rd \eta .
\end{split}
\end{equation}
Similarly, by the Cauchy-Schwarz inequality, we have
\begin{equation}\label{eq:Del3.-1}
 \begin{split}
& \left\|\iint 
  \eta^{(2p+n)/2-1} F(\circ) 
\nabla_\theta\bar{\pi}(\bullet) h_i^2(\bullet)\rd \theta   \rd \eta \right\|^2 \\
& \leq
 \iint   \eta^{(2p+n)/2-2} \mathcal{F}^2(\circ) f(\circ)
 \|\nabla_\theta\bar{\pi}(\bullet)/\bar{\pi}(\bullet)\|^2 \bar{\pi}(\bullet)h_i^2(\bullet) \rd \theta   \rd \eta \\
 &\qquad \times
  \iint   \eta^{(2p+n)/2} f(\circ) \bar{\pi}(\bullet)h_i^2(\bullet) \rd \theta   \rd \eta \\
& \leq
 \iint   \eta^{(2p+n)/2-2} \mathcal{F}^2(\circ) f(\circ)
 \|\nabla_\theta\bar{\pi}(\bullet)/\bar{\pi}(\bullet)\|^2 \bar{\pi}(\bullet) \rd \theta   \rd \eta \\
 &\qquad \times
  \iint   \eta^{(2p+n)/2} f(\circ) \bar{\pi}(\bullet)h_i^2(\bullet) \rd \theta   \rd \eta ,
\end{split}
\end{equation}
where the second inequality follows from the fact $h_i^2\leq 1$.
Hence, by \eqref{eq:Del2.-1} and \eqref{eq:Del3.-1} with  \eqref{eq:deldel3},
\begin{align*}
\sup_i\Delta_{2i}\leq 2\iint   \eta^{(2p+n)/2-2} \mathcal{F}^2(\circ) f(\circ)
 \left\|\frac{\nabla_\theta\bar{\pi}(\bullet)}{\bar{\pi}(\bullet)}\right\|^2 \bar{\pi}(\bullet) \rd \theta   \rd \eta.
\end{align*}
By the relationship
\begin{equation}\label{eq:nablathetapi}
\|\nabla_\theta \bar{\pi}(\eta\|\theta\|^2)\|^2
=4\eta^2\|\theta\|^2\{\bar{\pi}'(\eta\|\theta\|^2)\}^2,
\end{equation}
we have
\begin{align}
\frac{ \sup_i\Delta_{2i}}{2}& =\iint   \eta^{(2p+n)/2-2} \mathcal{F}^2(\circ) f(\circ)
\left\|\frac{\nabla_\theta\bar{\pi}(\bullet)}{\bar{\pi}(\bullet)}\right\|^2 \bar{\pi}(\bullet) \rd \theta \rd \eta \notag\\
 &=4\iint   \eta^{(2p+n)/2-1} \mathcal{F}^2(\eta\{\|z-\theta\|^2+1\})
 f(\eta\{\|z-\theta\|^2+1\})\notag\\\
&\qquad\times \eta\|\theta\|^2\left\{\frac{\bar{\pi}'(\eta\|\theta\|^2)}{\bar{\pi}(\eta\|\theta\|^2)}\right\}^2
 \bar{\pi}(\eta\|\theta\|^2) \rd \theta   \rd \eta \notag\\
&\leq 4 \Pi^2
\iint  \eta^{(2p+n)/2-1} \mathcal{F}^2(\eta\{\|z-\theta\|^2+1\})
\notag\ \\
 &\qquad\times f(\eta\{\|z-\theta\|^2+1\})\frac{\bar{\pi}(\eta\|\theta\|^2)}{\eta\|\theta\|^2}
\rd \theta   \rd \eta , \label{eq:delta22}
\end{align}
where
\begin{equation}\label{eq:Pi}
 \Pi=\max_{\lambda\in\mathbb{R}_+} \frac{\lambda|\bar{\pi}'(\lambda)|}{\bar{\pi}(\lambda)}=
  \max_{\lambda\in\mathbb{R}_+}\left|1-\frac{p}{2}+\frac{\lambda\pi'(\lambda)}{\pi(\lambda)}\right|,
\end{equation}
is bounded by Part \ref{lem:pi:1} of Lemma \ref{lem:pi}.
Further, by \eqref{eq:delta22}, we have 
\begin{equation}\label{eq:last.sec:subsub1}
 \begin{split}
 &\frac{1 }{8\Pi^2}\int  \sup_i\Delta_{2i}\rd z \\
 &\leq
 \iiint   \eta^{(p+n)/2-1} \mathcal{F}^2(\eta\{\|z\|^2+1\})
 f(\eta\{\|z\|^2+1\}) 
\frac{\bar{\pi}(\|\mu\|^2)}{\|\mu\|^2} \rd \mu   \rd \eta  \rd z \\
 &=
 c_pA_3 \iint  w^{p/2-1} \eta^{(p+n)/2-1} \mathcal{F}^2(\eta\{w+1\})f(\eta\{w+1\})
   \rd \eta  \rd w \\
 &=
 c_pA_3 \int_0^\infty  \frac{w^{p/2-1} \rd w}{(1+w)^{(p+n)/2}}\int_0^\infty t^{(p+n)/2-1}\mathcal{F}^2(t)f(t) \rd t \\
 &= c_p A_2 A_3 B (p/2,n/2), 
\end{split}
\end{equation}
where $A_2$ given by \eqref{proof.A_2} is bounded and
\begin{align}
 A_3=\int_{\mathbb{R}^p} \frac{\bar{\pi}(\|\mu\|^2)}{\|\mu\|^2} \rd \mu
 =\int_0^\infty \frac{\pi(\lambda)}{\lambda} \rd \lambda\label{proof.A_3}
\end{align}
is also bounded as shown in Parts \ref{lem:pi:2} and \ref{lem:pi:4} of Lemma \ref{lem:pi}.
\subsubsection{$\Delta_{2i}$ under Assumption \ref{AA2} with $-1/2<\alpha\leq 0$}
\label{sec:subsub2}
Let
\begin{align*}
 k(\lambda)=\lambda^{1/2}I_{[0,1]}(\lambda)+ I_{(1,\infty)}(\lambda).
\end{align*}
Note that $0\leq k(\lambda)\leq 1$ and $1/k(\lambda)\geq 1$ for any $\lambda\geq 0$.
Then, by Cauchy-Schwarz inequality, we have
\begin{equation}\label{eq:Del2}
 \begin{split}
& \left\|\iint 
  \eta^{(2p+n)/2-1} F(\circ) 
\nabla_\theta\bar{\pi}(\bullet) \rd \theta   \rd \eta \right\|^2 \\
& \leq
 \iint   \eta^{(2p+n)/2-2} \mathcal{F}^2(\circ) f(\circ)
  k(\bullet) \left\|\frac{\nabla_\theta\bar{\pi}(\bullet)}{\bar{\pi}(\bullet)}\right\|^2
  \bar{\pi}(\bullet) \rd \theta   \rd \eta \\
 &\qquad \times
 \iint   \eta^{(2p+n)/2} f(\circ) \frac{\bar{\pi}(\bullet)}{k(\bullet)} \rd \theta   \rd \eta \\
& \leq
\mathcal{J}_1(f,\pi) \iint   \eta^{(2p+n)/2-2} \mathcal{F}^2(\circ) f(\circ)
  k(\bullet) \left\|\frac{\nabla_\theta\bar{\pi}(\bullet)}{\bar{\pi}(\bullet)}\right\|^2
  \bar{\pi}(\bullet) \rd \theta   \rd \eta \\
 &\qquad \times
 \iint   \eta^{(2p+n)/2} f(\circ) \bar{\pi}(\bullet) \rd \theta   \rd \eta,
\end{split}
\end{equation}
where the second inequality with the constant $\mathcal{J}_1(f,\pi)$ follows
from Lemma \ref{lem:f_around_origin}, provided below in Appendix \ref{sec:finallemmas}.
Similarly, by the Cauchy-Schwarz inequality, we have
\begin{equation}\label{eq:Del2.00}
 \begin{split}
& \left\|\iint 
  \eta^{(2p+n)/2-1} F(\circ) 
\nabla_\theta\bar{\pi}(\bullet) h_i^2(\bullet)\rd \theta   \rd \eta \right\|^2 \\
& \leq
 \iint   \eta^{(2p+n)/2-2} \mathcal{F}^2(\circ) f(\circ)
 k(\bullet)\left\|\frac{\nabla_\theta\bar{\pi}(\bullet)}{\bar{\pi}(\bullet)}\right\|^2 \bar{\pi}(\bullet) h_i^2(\bullet)\rd \theta   \rd \eta \\
 &\qquad \times
 \iint   \eta^{(2p+n)/2} f(\circ) \frac{\bar{\pi}(\bullet)}{k(\bullet)} h_i^2(\bullet)\rd \theta   \rd \eta \\
& \leq
\mathcal{J}_2(f,\pi) \iint   \eta^{(2p+n)/2-2} \mathcal{F}^2(\circ) f(\circ)
  k(\bullet)
  \left\|\frac{\nabla_\theta\bar{\pi}(\bullet)}{\bar{\pi}(\bullet)}\right\|^2
  \bar{\pi}(\bullet) \rd \theta   \rd \eta \\
 &\qquad \times
 \iint   \eta^{(2p+n)/2} f(\circ) \bar{\pi}_i(\bullet) \rd \theta   \rd \eta,
\end{split}
\end{equation}
where the second inequality with the constant $\mathcal{J}_2(f,\pi)$ follows from Lemma \ref{lem:f_around_origin}, provided below in Appendix \ref{sec:finallemmas},
and from the fact $h_i^2\leq 1$.
Hence, by \eqref{eq:nablathetapi}, \eqref{eq:Del2}, \eqref{eq:Del2.00} and with
$\Pi$ 
given by \eqref{eq:Pi}, we have 
\begin{align*}
\frac{\sup_i\Delta_{2i}}{\mathcal{J}_1(f,\pi)+\mathcal{J}_2(f,\pi)}
 &\leq \iint   \eta^{(2p+n)/2-2} \mathcal{F}^2(\circ) f(\circ)
 k(\bullet)
 \left\|\frac{\nabla_\theta\bar{\pi}(\bullet)}{\bar{\pi}(\bullet)}\right\|^2
 \bar{\pi}(\bullet) \rd \theta   \rd \eta \\
 &\leq 4\Pi^2 \iint   \eta^{(2p+n)/2-1} \mathcal{F}^2(\circ) f(\circ)
 k(\bullet)\frac{\bar{\pi}(\bullet)}{\eta\|\theta\|^2} \rd \theta   \rd \eta. 
\end{align*}
Therefore we have
\begin{align}
 &\frac{1 }{4\Pi^2\left\{\mathcal{J}_1(f,\pi)+\mathcal{J}_2(f,\pi)\right\}}
 \int  \sup_i\Delta_{2i}\rd z \notag \\
 &\leq
 \iiint   \eta^{(p+n)/2-1} \mathcal{F}^2(\eta\{\|z\|^2+1\})
 f(\eta\{\|z\|^2+1\}) 
k(\|\mu\|^2)\frac{\bar{\pi}(\|\mu\|^2)}{\|\mu\|^2} \rd \mu   \rd \eta  \rd z \notag \\
 &=
 A_4c_p \iint  w^{p/2-1} \eta^{(p+n)/2-1} \mathcal{F}^2(\eta\{w+1\})f(\eta\{w+1\})
   \rd \eta  \rd w \label{eq:last.sec:subsub2} \\
 &=
 A_4c_p \int  \frac{w^{p/2-1} \rd w}{(1+w)^{(p+n)/2}}\int_0^\infty t^{(p+n)/2-1}\mathcal{F}^2(t)f(t) \rd t \notag\\
 &=
 A_2A_4 c_p B (p/2,n/2), \notag 
\end{align}
where $A_2$ given by \eqref{proof.A_2} is bounded and
\begin{align*}
 A_4&=\int_{\mathbb{R}^p} k(\|\mu\|^2)\frac{\bar{\pi}(\|\mu\|^2)}{\|\mu\|^2} \rd \mu \\
 &=\int_{\|\mu\|\leq 1} \frac{\bar{\pi}(\|\mu\|^2)}{\|\mu\|} \rd \mu
 +\int_{\|\mu\|> 1} \frac{\bar{\pi}(\|\mu\|^2)}{\|\mu\|^2} \rd \mu \\
 &=\int_0^1 \frac{\pi(\lambda)}{\lambda^{1/2}} \rd \lambda
+ \int_1^\infty \frac{\pi(\lambda)}{\lambda} \rd \lambda
\end{align*}
is bounded as shown in Parts \ref{lem:pi:2} and \ref{lem:pi:4} of Lemma \ref{lem:pi}.

The proof of Theorem \ref{thm:mainmain} Case I is thus completed by applying the dominated convergence
theorem to $\ndiff(\delta_\pi,\delta_{\pi i};\pi_i)$ as noted above.

\section{Proof for case II}
\label{sec:caseII}
Recall, under Assumption \ref{AA3.2}, $\pi$ and $\bar{\pi}$ satisfy
\begin{align*}
 \lambda\frac{\pi'(\lambda)}{\pi(\lambda)}=\kappa(\lambda), \text{ and }\lambda\frac{\bar{\pi}'(\lambda)}{\bar{\pi}(\lambda)}=1-\frac{p}{2}+\kappa(\lambda),
\end{align*}
where $\kappa(\lambda)\to 0$ as $\lambda\to\infty$.

With $\kappa(\lambda)$, we have
\begin{equation}\label{eq:nabla_theta_pi}
 \begin{split}
 \nabla_\theta\bar{\pi}(\eta\|\theta\|^2)&=
  2\eta\theta\bar{\pi}'(\eta\|\theta\|^2)\\
 &=2\eta\theta\left\{(1-p/2)\frac{\bar{\pi}(\eta\|\theta\|^2)}{\eta\|\theta\|^2}+
 \frac{\bar{\pi}(\eta\|\theta\|^2)}{\eta\|\theta\|^2}\kappa(\eta\|\theta\|^2)\right\}.
\end{split}
\end{equation}
By Lemma \ref{lem:general_integration_by_parts} and the relationship
\eqref{eq:nabla_theta_pi}, the integral included in \eqref{diff.general}
is rewritten as
\begin{align*}
 & -z^\T\iint   \eta^{(2p+n)/2-1} F(\circ)
 \nabla_\theta\bar{\pi}(\bullet) \rd \theta   \rd \eta \\
 &\quad = \left(\frac{p-2}{n+2}-\frac{n+p}{n+2}\right)
 z^\T\iint   \eta^{(2p+n)/2-1} F(\circ)
 \nabla_\theta\bar{\pi}(\bullet) \rd \theta   \rd \eta \\
 &\quad =\frac{p-2}{n+2}\iint   \eta^{(2p+n)/2} f(\circ)
 \bar{\pi}(\bullet) \rd \theta   \rd \eta \\
&\quad \quad -\frac{(n+p)(p-2)}{n+2}\iint   \eta^{(2p+n)/2-1} F(\circ)
 \bar{\pi}(\bullet) \rd \theta   \rd \eta \\
 &\quad \quad - \frac{n+p}{n+2}z^\T\iint   \eta^{(2p+n)/2-1} F(\circ)
 \nabla_\theta\bar{\pi}(\bullet) \rd \theta   \rd \eta \\
 &\quad =\frac{p-2}{n+2}\iint   \eta^{(2p+n)/2} f(\circ)
 \bar{\pi}(\bullet) \rd \theta   \rd \eta \\
&\quad \quad +\frac{(n+p)(p-2)}{n+2}\iint  \frac{z^\T\theta-\|\theta\|^2}{\|\theta\|^2} \eta^{(2p+n)/2-1} F(\circ)
 \bar{\pi}(\bullet) \rd \theta   \rd \eta \\
&\quad \quad -2\frac{n+p}{n+2}z^\T\iint  \theta\eta^{(2p+n)/2-1} F(\circ)
 \frac{\kappa(\bullet)\bar{\pi}(\bullet)}{\|\theta\|^2} \rd \theta   \rd \eta 
\end{align*}
where, again, with the notation  
\begin{align*}
 \bullet=\eta\|\theta\|^2, \quad \circ=\eta(\|z-\theta\|^2+1).
\end{align*}
Similarly, by Lemma \ref{lem:general_integration_by_parts} and the relationship
\eqref{eq:nabla_theta_pi}, the integral included in \eqref{diff.general}
is rewritten as
\begin{align*}
 & -z^\T\iint   \eta^{(2p+n)/2-1} F(\circ)
 \nabla_\theta\bar{\pi}_i(\bullet) \rd \theta   \rd \eta \\
 &\quad =\frac{p-2}{n+2}\iint   \eta^{(2p+n)/2} f(\circ)
 \bar{\pi}_i(\bullet) \rd \theta   \rd \eta \\
&\quad \quad -\frac{(n+p)(p-2)}{n+2}\iint   \eta^{(2p+n)/2-1} F(\circ)
 \bar{\pi}_i(\bullet) \rd \theta   \rd \eta \\
 &\quad \quad - \frac{n+p}{n+2}z^\T\iint   \eta^{(2p+n)/2-1} F(\circ)
 \nabla_\theta\bar{\pi}_i(\bullet) \rd \theta   \rd \eta \\
 &\quad =\frac{p-2}{n+2}\iint   \eta^{(2p+n)/2} f(\circ)
 \bar{\pi}_i(\bullet) \rd \theta   \rd \eta \\
&\quad \quad +\frac{(n+p)(p-2)}{n+2}\iint  \frac{z^\T\theta-\|\theta\|^2}{\|\theta\|^2} \eta^{(2p+n)/2-1} F(\circ)
 \bar{\pi}_i(\bullet) \rd \theta   \rd \eta \\
&\quad \quad -2\frac{n+p}{n+2}z^\T\iint  \theta\eta^{(2p+n)/2-1} F(\circ)
 \frac{\kappa(\bullet)\bar{\pi}_i(\bullet)}{\|\theta\|^2} \rd \theta   \rd \eta \\
 &\quad \quad - \frac{n+p}{n+2}z^\T\iint   \eta^{(2p+n)/2-1} F(\circ)
 \bar{\pi}(\bullet)\nabla_\theta h_i^2(\bullet) \rd \theta   \rd \eta .
\end{align*}
Then $\oldiff(z;\delta_\pi,\delta_{\pi i};\pi_i)$ given by \eqref{diff.general}
is rewritten as
\begin{align*}
& \oldiff(z;\delta_\pi,\delta_{\pi i};\pi_i) \\ 
 &=\frac{c_n}{\|z\|^2}\frac{(n+p)^2}{(n+2)^2}\left\{
z^\T\frac{\iint   \eta^{(2p+n)/2-1} F(\circ)
 \bar{\pi}(\bullet)\nabla_\theta h_i^2(\bullet) \rd \theta   \rd \eta} 
{\iint  \eta^{(2p+n)/2} f(\circ) \bar{\pi}_i(\bullet) \rd \theta   \rd \eta} \right. \\
&\quad \left. +(p-2) \frac{
 \iint  (z^\T\theta/\|\theta\|^2-1) \eta^{(2p+n)/2-1} F(\circ)
 \bar{\pi}(\bullet) \rd \theta   \rd \eta }
 {\iint  \eta^{(2p+n)/2} f(\circ) \bar{\pi}(\bullet) \rd \theta   \rd \eta} \right. \\
 &\quad \left. -
 (p-2) \frac{
 \iint  (z^\T\theta/\|\theta\|^2-1) \eta^{(2p+n)/2-1} F(\circ)
 \bar{\pi}_i(\bullet) \rd \theta   \rd \eta }
 {\iint  \eta^{(2p+n)/2} f(\circ) \bar{\pi}_i(\bullet) \rd \theta   \rd \eta} \right. \\
& \quad \left. -
2\frac{z^\T\iint  \theta\eta^{(2p+n)/2-1} F(\circ)
 \{\kappa(\bullet)\bar{\pi}(\bullet)/\|\theta\|^2\} \rd \theta   \rd \eta}
 {\iint  \eta^{(2p+n)/2} f(\circ) \bar{\pi}(\bullet) \rd \theta   \rd \eta} \right. \\
& \quad \left. +
2\frac{z^\T\iint  \theta\eta^{(2p+n)/2-1} F(\circ)
 \{\kappa(\bullet)\bar{\pi}_i(\bullet)/\|\theta\|^2\} \rd \theta   \rd \eta}
 {\iint  \eta^{(2p+n)/2} f(\circ) \bar{\pi}_i(\bullet) \rd \theta   \rd \eta} 
\right\}^2 \\
&\qquad\times \iint\eta^{(2p+n)/2} f(\circ) \bar{\pi}_i(\bullet) \rd \theta   \rd \eta .
\end{align*}
By the triangle inequality and the fact $h_i^2\leq 1$,
\begin{align*}
 \oldiff(z;\delta_\pi,\delta_{\pi i};\pi_i)  
\leq 2\frac{c_n(n+p)^2}{(n+2)^2}\left\{\Delta_{1i}+(p-2)^2\Delta_{3i}
 +4\Delta_{4i}\right\},
\end{align*}
where 
\begin{align}
 \Delta_{1i}
& =
\frac{\left\| \iint   \eta^{(2p+n)/2-1} F(\circ) 
\bar{\pi}(\bullet)\nabla_\theta h_i^2(\bullet) \rd \theta   \rd \eta\right\|^2}{\iint\eta^{(2p+n)/2} f(\circ) 
\bar{\pi}_i(\bullet) \rd \theta   \rd \eta }, \notag \\
 \Delta_{3i}
& =
 \frac{1}{\|z\|^2}
 \frac{\left\{\iint 
(z^\T\theta/\|\theta\|^2-1) \eta^{(2p+n)/2-1} F(\circ) 
\bar{\pi}(\bullet) \rd \theta   \rd \eta\right\}^2}{\iint \eta^{(2p+n)/2} f(\circ) 
 \bar{\pi}(\bullet) \rd \theta   \rd \eta} \notag\\ 
& \quad +
 \frac{1}{\|z\|^2}
 \frac{\left\{\iint 
(z^\T\theta/\|\theta\|^2-1) \eta^{(2p+n)/2-1} F(\circ) 
\bar{\pi}_i(\bullet) \rd \theta   \rd \eta\right\}^2}{\iint \eta^{(2p+n)/2} f(\circ) 
 \bar{\pi}_i(\bullet) \rd \theta   \rd \eta}, \label{eq:deldel5}\\
 \Delta_{4i}
 &=
 \frac{\left\| \iint  \theta\eta^{(2p+n)/2-1} F(\circ) \kappa(\bullet)\bar{\pi}(\bullet)\|\theta\|^{-2} \rd \theta  \rd \eta\right\|^2}{\iint\eta^{(2p+n)/2} f(\circ) 
\bar{\pi}(\bullet) \rd \theta   \rd \eta } \notag\\ 
 &\quad +
 \frac{\left\| \iint  \theta\eta^{(2p+n)/2-1} F(\circ) \kappa(\bullet)\bar{\pi}_i(\bullet)\|\theta\|^{-2} \rd \theta  \rd \eta\right\|^2}{\iint\eta^{(2p+n)/2} f(\circ) 
\bar{\pi}_i(\bullet) \rd \theta   \rd \eta }. \label{eq:deldel7}
\end{align}
For $\Delta_{1i}$, as seen in Section \ref{sec:del1},
we have $\int \sup_i\Delta_{1i}\rd z<\infty$.
We will show integrability $\int \sup_i\Delta_{3i}\rd z<\infty$ 
and integrability $\int \sup_i\Delta_{4i}\rd z<\infty$ in Sub-sections
\ref{sec:delta45} and \ref{sec:delta67}, respectively.
 
  \subsection{$\Delta_{3i}$}
\label{sec:delta45}
  Note the inequality
 \begin{align*}
\left|\frac{z^\T\theta}{\|\theta\|^2} -1 \right|
=\left|\frac{(z-\theta)^\T\theta}{\|\theta\|^2}  \right|
\leq \frac{\|z-\theta\|}{\|\theta\|} 
\leq \frac{\sqrt{\|z-\theta\|^2+1}}{\|\theta\|}.
 \end{align*}
Then, in the first and second terms of \eqref{eq:deldel5}, we have
\begin{align*}
  &\iint  \left|\frac{z^\T\theta-\|\theta\|^2}{\|\theta\|^2}\right| \eta^{(2p+n)/2-1} F(\eta\{\|z-\theta\|^2+1\})
 \bar{\pi}(\eta\|\theta\|^2) \rd \theta   \rd \eta \\
 &\leq
 \iint \frac{\eta^{(2p+n)/2-1}}{\eta^{1/2}\|\theta\|}
\{\eta(\|z-\theta\|^2+1)\}^{1/2}F(\eta\{\|z-\theta\|^2+1\})
 \bar{\pi}(\eta\|\theta\|^2) \rd \theta   \rd \eta \\
 &=
 \iint \frac{\eta^{(2p+n)/2-1}}{\eta^{1/2}\|\theta\|}
 \tilde{\mathcal{F}}(\eta\{\|z-\theta\|^2+1\})f(\eta\{\|z-\theta\|^2+1\})
 \bar{\pi}(\eta\|\theta\|^2) \rd \theta   \rd \eta
\end{align*}
where
\begin{align*}
 \tilde{\mathcal{F}}(t)=t^{1/2}\mathcal{F}(t)=t^{1/2}\frac{F(t)}{f(t)}
\end{align*}
and
\begin{align*}
  &\iint  \left|\frac{z^\T\theta-\|\theta\|^2}{\|\theta\|^2}\right| \eta^{(2p+n)/2-1} F(\eta\{\|z-\theta\|^2+1\})
 \bar{\pi}_i(\eta\|\theta\|^2) \rd \theta   \rd \eta \\
 &\leq
 \iint \frac{\eta^{(2p+n)/2-1}}{\eta^{1/2}\|\theta\|}
 \tilde{\mathcal{F}}(\eta\{\|z-\theta\|^2+1\}) f(\eta\{\|z-\theta\|^2+1\})
 \bar{\pi}_i(\eta\|\theta\|^2) \rd \theta   \rd \eta.
\end{align*}
Under Assumption \ref{AA2} on $\pi$ with $\alpha>0$,
applying the same technique used in Sub-Section \ref{sec:subsub1},
the integrability of
\begin{align*}
B_1 & =\iiint \frac{\eta^{(2p+n)/2-2}}{\|z\|^2}   \tilde{\mathcal{F}}^2(\eta\{\|z-\theta\|^2+1\})
 f(\eta\{\|z-\theta\|^2+1\}) \\
&\qquad \times \frac{\bar{\pi}(\eta\|\theta\|^2)}{\eta\|\theta\|^2}
 \rd \theta   \rd \eta  \rd z 
\end{align*}
implies the integrability of $ \int  \sup_i\Delta_{3i}\rd z$. 
The integrability of $B_1$ is shown as follows;
\begin{align}
B_1 &=
 \iiint \frac{\eta^{n/2-1}}{\|y\|^2}   \tilde{\mathcal{F}}^2(\|y-\mu\|^2+\eta)
 f(\|y-\mu\|^2+\eta)
 \frac{\bar{\pi}(\|\mu\|^2)}{\|\mu\|^2}
  \rd \mu   \rd \eta  \rd y \notag\\
 &=
 \int_{\mathbb{R}^p}\left(\int_{\mathbb{R}^p}\frac{f_\star(\|y-\mu\|^2)}{\|y\|^2}\rd y\right)
  \frac{\bar{\pi}(\|\mu\|^2)}{\|\mu\|^2} \rd \mu \notag\\
 &\leq \mathcal{Q}_f
 \int_{\mathbb{R}^p}
\min(1,\|\mu\|^{-2})\frac{\bar{\pi}(\|\mu\|^2)}{\|\mu\|^2} \rd \mu \label{B1B1}\\
 &= \mathcal{Q}_f
 \left\{\int_{\|\mu\|\leq 1}\frac{\bar{\pi}(\|\mu\|^2)}{\|\mu\|^2} \rd \mu +
 \int_{\|\mu\|> 1}\frac{\bar{\pi}(\|\mu\|^2)}{\|\mu\|^4} \rd \mu \right\} \notag\\
 &= \mathcal{Q}_f
 \left\{\int_0^1 \frac{\pi(\lambda)}{\lambda} \rd \lambda +
\int_1^\infty \frac{\pi(\lambda)}{\lambda^2} \rd \lambda  \right\} <\infty \notag
\end{align}
where
 \begin{align*}
   f_\star(t)=\int_0^\infty \eta^{n/2-1}\tilde{\mathcal{F}}^2(t+\eta)f(t+\eta)\rd \eta,
 \end{align*}
 the inequality with $\mathcal{Q}_f$ follows from Part \ref{lem:f.2} of Lemma \ref{lem:f}
and the integrability of the right-hand side follows from
 Parts \ref{lem:pi:3} and \ref{lem:pi:5} of Lemma \ref{lem:pi}.

Under Assumption \ref{AA2} on $\pi$ with $-1/2<\alpha\leq 0$,
applying the same technique used in Sub-Section \ref{sec:subsub2},
the integrability of
\begin{align*}
B_2&= \iiint \frac{\eta^{(2p+n)/2-2}}{\|z\|^2}   \tilde{\mathcal{F}}^2(\eta\{\|z-\theta\|^2+1\})
 f(\eta\{\|z-\theta\|^2+1\}) \\
&\qquad \times \frac{k(\eta\|\theta\|^2)\bar{\pi}(\eta\|\theta\|^2)}{\eta\|\theta\|^2}
 \rd \theta   \rd \eta  \rd z 
\end{align*}
where $k(\lambda)=\lambda^{1/2}I_{[0,1]}(\lambda) + I_{(1,\infty)}(\lambda)$,
implies the integrability of $\int  \sup_i\Delta_{3i}\rd z$.
As in \eqref{B1B1}, $B_2$ is given by
\begin{align*}
B_2  &\leq \mathcal{Q}_f
 \int_{\mathbb{R}^p}
\min(1,\|\mu\|^{-2})\frac{k(\|\mu\|^2)\bar{\pi}(\|\mu\|^2)}{\|\mu\|^2} \rd \mu \\
 &= \mathcal{Q}_f
 \left\{
\int_0^1 \frac{\pi(\lambda)}{\lambda^{1/2}} \rd \lambda +
\int_1^\infty \frac{\pi(\lambda)}{\lambda^2} \rd \lambda 
 \right\} <\infty
\end{align*}
which is bounded by Parts \ref{lem:pi:2} and \ref{lem:pi:5} of Lemma \ref{lem:pi}.

  \subsection{$\Delta_{4i}$}
\label{sec:delta67}
Under Assumption \ref{AA2} on $\pi$ with $\alpha>0$,
applying the same technique used in Sub-Section \ref{sec:subsub1},
the integrability of
\begin{align*}
B_3 & =\iiint \eta^{(2p+n)/2-1}  \mathcal{F}^2(\eta\{\|z-\theta\|^2+1\})
 f(\eta\{\|z-\theta\|^2+1\}) \\
&\qquad \times \frac{\kappa^2(\eta\|\theta\|^2)\bar{\pi}(\eta\|\theta\|^2)}{\eta\|\theta\|^2}
 \rd \theta   \rd \eta  \rd z 
\end{align*}
implies the integrability of $ \int  \sup_i\Delta_{4i}\rd z$. 
The integrability of $B_3$ is shown as follows;
\begin{align*}
B_3 & =c_pA_2B(p/2,n/2)\int 
\frac{\kappa^2(\|\mu\|^2)\bar{\pi}(\|\mu\|^2)}{\|\mu\|^2}
 \rd \mu \\
 &=c_pA_2B(p/2,n/2)\left\{
\sup_{\lambda\in(0,1)}\kappa^2(\lambda)\int_0^1 \frac{\pi(\lambda)}{\lambda} \rd \mu 
 +\int_1^\infty \frac{\pi(\lambda)\kappa^2(\lambda)}{\lambda} \rd \mu
 \right\},
\end{align*}
where $A_2$ is given by \eqref{proof.A_2},
the first term is bounded by Parts \ref{lem:pi:1} and \ref{lem:pi:3} of Lemma \ref{lem:pi}
and the second term is bounded by Part \ref{lem:pi:8} of Lemma \ref{lem:pi}.

Under Assumption \ref{AA2} on $\pi$ with $-1/2<\alpha\leq 0$,
applying the same technique used in Sub-Section \ref{sec:subsub2}
the integrability of
\begin{align*}
B_4 & =\iiint \eta^{(2p+n)/2-1}  \mathcal{F}^2(\eta\{\|z-\theta\|^2+1\})
 f(\eta\{\|z-\theta\|^2+1\}) \\
 &\qquad \times
 \frac{k(\eta\|\theta\|^2)\kappa^2(\eta\|\theta\|^2)\bar{\pi}(\eta\|\theta\|^2)}{\eta\|\theta\|^2}
 \rd \theta   \rd \eta  \rd z 
\end{align*}
where $k(\lambda)=\lambda^{1/2}I_{[0,1]}(\lambda) + I_{(1,\infty)}(\lambda)$
implies the integrability of $\int  \sup_i\Delta_{4i}\rd z$.
The integrability of $B_3$ is shown as follows;
\begin{align*}
B_4 & =c_pA_2B(p/2,n/2)\int 
\frac{k(\|\mu\|^2)\kappa^2(\|\mu\|^2)\bar{\pi}(\|\mu\|^2)}{\|\mu\|^2}
 \rd \mu \\
 &=c_pA_2B(p/2,n/2)\left\{
\sup_{\lambda\in(0,1)}\kappa^2(\lambda)\int_0^1 \frac{\pi(\lambda)}{\lambda^{1/2}} \rd \mu 
 +\int_1^\infty \frac{\pi(\lambda)\kappa^2(\lambda)}{\lambda} \rd \mu \right\},
\end{align*}
where $A_2$ is given by \eqref{proof.A_2},
the first term is bounded by Parts \ref{lem:pi:1} and \ref{lem:pi:2} of Lemma \ref{lem:pi}
and the second term is bounded by Part \ref{lem:pi:8} of Lemma \ref{lem:pi}.

  \section{Additional Lemmas used in Sections \ref{sec:caseI} and \ref{sec:caseII}}
\label{sec:finallemmas}
  Let
   \begin{align}
    &J(f,\pi,z)\label{eq:Mfpi}\\
& =   \frac{\iint_{\eta\|\theta\|^2\leq 1} 
  \eta^{(2p+n)/2} \{\eta\|\theta\|^2\}^{-1/2}\bar{\pi}(\eta\|\theta\|^2) f(\eta\{\|z-\theta\|^2+1\}) 
 \rd \theta   \rd \eta}{\iint_{\eta\|\theta\|^2\leq 1}  
  \eta^{(2p+n)/2} \bar{\pi}(\eta\|\theta\|^2) f(\eta\{\|z-\theta\|^2+1\}) 
  \rd \theta   \rd \eta}.  \notag
   \end{align}
Then we have a following result.
  \begin{lemma}\label{lem:f_around_origin.0}
   Suppose Assumptions \ref{FF1}--\ref{FF3} on $f$ hold.
   Assume Assumptions \ref{AA2} on $\pi$ with $-1/2<\alpha\leq 0$. Then
\begin{equation}\label{eq:lem:f_around_origin.0}
J(f,\pi,z) \leq \mathcal{J}(f,\pi)<\infty
\end{equation} 
for any $z\in\mathbb{R}^p$, where
\begin{equation}\label{eq:mathcalMfpi}
 \begin{split}
  \mathcal{J}(f,\pi)&=2\frac{\alpha+1}{\alpha+1/2}
\frac{\max_{a\in\mathbb{R}_+}\varphi(a;\alpha+1/2)}
  {\min_{a\in\mathbb{R}_+}\varphi(a;\alpha+1)}
  \frac{\max_{\lambda\in[0,1]}\nu(\lambda)}{\min_{\lambda\in[0,1]}\nu(\lambda)}, \\
  \varphi(a;\gamma)&=\frac{\int_0^a  t^{(p+n)/2+\gamma} f(t)  \rd t }{\int_0^a  t^{(p+n)/2+\gamma} f_G(t)  \rd t }, \\
  f_G(t)&=(2\pi)^{-(p+n)/2}\exp(-t/2).
 \end{split}
\end{equation} 
  \end{lemma}
\begin{proof}
 By Assumptions \ref{AA2} on $\pi$,
\begin{equation}\label{MM1}
 J(f,\pi,z)\leq \frac{\max_{\lambda\in[0,1]}\nu(\lambda)}{\min_{\lambda\in[0,1]}\nu(\lambda)}
  J_1(f,\pi,z)
\end{equation}
 where
 \begin{align}
& J_1(f,\pi,z) \label{def:M1}\\
&= \frac{\iint_{\eta\|\theta\|^2\leq 1} 
  \eta^{(2p+n)/2} \{\eta\|\theta\|^2\}^{\alpha+(1-p)/2} f(\eta\{\|z-\theta\|^2+1\}) 
 \rd \theta   \rd \eta}{\iint_{\eta\|\theta\|^2\leq 1}  
  \eta^{(2p+n)/2} \{\eta\|\theta\|^2\}^{\alpha+(2-p)/2} f(\eta\{\|z-\theta\|^2+1\}) 
 \rd \theta   \rd \eta}\notag \\
  &=
  \frac{\int_{\mathbb{R}^p} \|\theta\|^{2\alpha+1-p}
  \left\{\int_0^{1/\|\theta\|^2}\eta^{(p+n+1)/2+\alpha} f(\eta\{\|z-\theta\|^2+1\}) \rd \eta\right\}
  \rd \theta   }
  {\int_{\mathbb{R}^p} \|\theta\|^{2\alpha+2-p}
  \left\{\int_0^{1/\|\theta\|^2}\eta^{(p+n+2)/2+\alpha} f(\eta\{\|z-\theta\|^2+1\}) \rd \eta\right\}
  \rd \theta }.\notag
 \end{align}
 Let $\gamma=\alpha+1$ for the denominator and $\alpha+1/2$ for the numerator
 of $J_1(f,\pi,z)$.
 By change of variables, the integral with respect to $\eta$ is rewritten as
\begin{align*}
& \int_0^{1/\|\theta\|^2} 
  \eta^{(p+n)/2+\gamma} f(\eta\{\|z-\theta\|^2+1\}) 
 \rd \eta \\
& =\{\|z-\theta\|^2+1\}^{-(p+n)/2-1-\gamma}\int_0^a
  t^{(p+n)/2+\gamma} f(t)  \rd t \\
 & =\varphi(a;\gamma)
\int_0^{1/\|\theta\|^2} 
  \eta^{(p+n)/2+\gamma}  f_G(\eta\{\|z-\theta\|^2+1\})  \rd \eta
\end{align*} 
 where $\varphi(a;\gamma)$ is defined by \eqref{eq:mathcalMfpi} and $a=\{\|z-\theta\|^2+1\}/\|\theta\|^2$.
 Note
\begin{align*}
 \lim_{a\to 0}\varphi(a;\gamma)=\frac{f(0)}{f_G(0)}\text{ and }\lim_{a\to \infty}
 \varphi(a;\gamma)=
\frac{\int_0^\infty  t^{(p+n)/2+\gamma} f(t)  \rd t }{\int_0^\infty  t^{(p+n)/2+\gamma} f_G(t)  \rd t },
\end{align*}
 which are both positive and bounded from the above under $0<\gamma\leq 1$ and
 under Assumptions \ref{FF1}--\ref{FF3} on $f$
and hence
 \begin{align*}
  \min_{a\in\mathbb{R}_+}\varphi(a;\gamma)>0
  \text{ and } \max_{a\in\mathbb{R}_+}\varphi(a;\gamma)<\infty,
 \end{align*}
under $0<\gamma\leq 1$. Therefore we have
 \begin{equation}\label{M1M2}
   J_1(f,\pi,z)
   \leq \frac{\max_{a\in\mathbb{R}_+}\varphi(a;\alpha+1/2)}
  {\min_{a\in\mathbb{R}_+}\varphi(a;\alpha+1)}J_2(f,\pi,z)
 \end{equation}
 where
\begin{align}
& J_2(f,\pi,z) \label{def:M2}\\ &= \frac{\iint_{\eta\|\theta\|^2\leq 1} 
  \eta^{(2p+n)/2} \{\eta\|\theta\|^2\}^{\alpha+(1-p)/2} f_G(\eta\{\|z-\theta\|^2+1\}) 
 \rd \theta   \rd \eta}{\iint_{\eta\|\theta\|^2\leq 1}  
  \eta^{(2p+n)/2} \{\eta\|\theta\|^2\}^{\alpha+(2-p)/2} f_G(\eta\{\|z-\theta\|^2+1\}) 
  \rd \theta   \rd \eta} \notag \\
&=\frac{\iint_{\|\mu\|^2\leq 1} 
  \eta^{(p+n)/2} \{\|\mu\|^2\}^{\alpha+(1-p)/2} \exp(-\|\eta^{1/2}z-\mu\|^2/2-\eta/2)
 \rd \mu  \rd \eta}
{\iint_{\|\mu\|^2\leq 1} 
  \eta^{(p+n)/2} \{\|\mu\|^2\}^{\alpha+(2-p)/2} \exp(-\|\eta^{1/2}z-\mu\|^2/2-\eta/2)
 \rd \mu  \rd \eta}.\notag
\end{align}
Note $\| \mu\|^2$ may be regarded as a non-central chi-square random variable
 with $p$ degrees of freedom and $\eta\|  z \|^2$ non-centrality parameter.
For 
 \begin{align*}
 a_j( \eta\|z\|^2 )=\frac{1}{\Gamma(p/2+j)2^{p/2+j}}\frac{(\eta\|  z \|^2/2)^j}{j!}\exp(-\eta\|  z \|^2/2),
 \end{align*}
 we have
 \begin{align}
&J_2(f,\pi,z) \notag\\
  &= \frac{\sum_{j=0}^\infty \int_0^\infty \eta^{(p+n)/2} a_j( \eta\|z\|^2 )\exp(-\eta/2)\rd \eta\int_0^1 r^{\alpha-1/2+j}\exp(-r/2) \rd r}
  {\sum_{j=0}^\infty \int_0^\infty \eta^{(p+n)/2} a_j( \eta\|z\|^2 )\exp(-\eta/2)\rd \eta\int_0^1
  r^{\alpha+j}\exp(-r/2) \rd r} \notag\\ 
  &=\frac{\sum_{j=0}^\infty \tilde{a}_j(\|z\|^2) E[R^{j-1/2}]}
  {\sum_{j=0}^\infty \tilde{a}_j(\|z\|^2) E[R^{j}]},\notag
 \end{align}
where the expected value is taken under the probability density given by
\begin{equation*}
 \frac{r^{\alpha}\exp(-r/2)I_{[0,1]}(r)}{\int_0^1r^{\alpha}\exp(-r/2) \rd r}
\end{equation*}
 and
\begin{align*}
 \tilde{a}_j(\|z\|^2)&=\int_0^\infty \eta^{(p+n)/2} a_j( \eta\|z\|^2 )\exp(-\eta/2)\rd \eta \\
 &=\frac{\Gamma((p+n)/2+j+1)2^{(p+n)/2+j+1}}{\Gamma(p/2+j)2^{p/2+j}}\frac{(\|  z \|^2/2)^j}{j!(\|  z \|^2+1)^{(p+n)/2+j+1}}.
\end{align*} 
 Since the correlation inequality gives
 \begin{align*}
 E[R^{-1/2}]\geq \frac{E[R^{1/2}]}{E[R]} \geq \frac{E[R^{3/2}]}{E[R^2]}\geq \dots,
 \end{align*}
we have
\begin{equation*}
\frac{\sum_{j=0}^\infty \tilde{a}_j(\|z\|^2) E[R^{j-1/2}]}
  {\sum_{j=0}^\infty \tilde{a}_j(\|z\|^2) E[R^{j}]}
  \leq E[R^{-1/2}]
  = \frac{\int_0^1r^{\alpha-1/2}\exp(-r/2) \rd r}{\int_0^1r^{\alpha}\exp(-r/2) \rd r}.
\end{equation*}
 For $0\leq r\leq 1$, we have
\begin{gather*}
1/2< \exp(-1/2)\leq \exp(-r/2)\leq 1, \\
 E[R^{-1/2}]\leq 2\frac{\alpha+1}{\alpha+1/2},
\end{gather*}
 and hence
\begin{equation}\label{M2C}
 J_2(f,\pi,z)\leq 2\frac{\alpha+1}{\alpha+1/2}, \text{ for any }z\in\mathbb{R}^p.
\end{equation}
 Finally, by \eqref{MM1}, \eqref{def:M1}, \eqref{M1M2}, \eqref{def:M2} and \eqref{M2C},
 we have
 \begin{align*}
J(f,\pi,z)\leq
 2\frac{\alpha+1}{\alpha+1/2}
\frac{\max_{a\in\mathbb{R}_+}\varphi(a;\alpha+1/2)}
  {\min_{a\in\mathbb{R}_+}\varphi(a;\alpha+1)}
  \frac{\max_{\lambda\in[0,1]}\nu(\lambda)}{\min_{\lambda\in[0,1]}\nu(\lambda)}.
 \end{align*}
\end{proof}

Using Lemma \ref{lem:f_around_origin.0}, we have the following result.
\begin{lemma}\label{lem:f_around_origin}
   Suppose Assumptions \ref{FF1}--\ref{FF3} on $f$ hold.
 Assume Assumptions \ref{AA2} on $\pi$ with $-1/2<\alpha\leq 0$.
 Let 
\begin{align*}
 k(\lambda)=\lambda^{1/2}I_{[0,1]}(\lambda)+ I_{(1,\infty)}(\lambda).
\end{align*}
\begin{enumerate}
 \item\label{lem:f_around_origin.part.1} Then 
\begin{equation}\label{eq:lem:f_around_origin.part.1}
 \begin{split}
  & \frac{ \iint   \eta^{(2p+n)/2} f(\eta\{\|z-\theta\|^2+1\})
  \{\bar{\pi}(\eta\|\theta\|^2)/k(\eta\|\theta\|^2)\} \rd \theta   \rd \eta}
 {\iint  \eta^{(2p+n)/2} f(\eta\{\|z-\theta\|^2+1\}) \bar{\pi}(\eta\|\theta\|^2) \rd \theta   \rd \eta}\\
&\leq \mathcal{J}_1(f,\pi),
 \end{split}
\end{equation} 
 where 
\begin{equation}
\mathcal{J}_1(f,\pi)=\mathcal{J}(f,\pi)+1
\end{equation}
and $\mathcal{J}(f,\pi)$ is given by \eqref{eq:mathcalMfpi} of Lemma \ref{lem:f_around_origin.0}.
 \item\label{lem:f_around_origin.part.2}  We have
\begin{equation}\label{eq:lem:f_around_origin.part.2}
 \begin{split}
  & \frac{ \iint   \eta^{(2p+n)/2} f(\eta\{\|z-\theta\|^2+1\})
  \{\bar{\pi}_i(\eta\|\theta\|^2)/k(\eta\|\theta\|^2)\} \rd \theta   \rd \eta}
 {\iint  \eta^{(2p+n)/2} f(\eta\{\|z-\theta\|^2+1\}) \bar{\pi}_i(\eta\|\theta\|^2) \rd \theta   \rd \eta} \\
&\leq \mathcal{J}_2(f,\pi),
\end{split}
\end{equation}
 where
 \begin{align*}
  \mathcal{J}_2(f,\pi)=64\mathcal{J}(f,\pi)+1.
 \end{align*}
\end{enumerate}
\end{lemma}
\begin{proof}
 Let $\mathcal{R}=\{(\theta,\eta):\eta\|\theta\|^2\leq 1\}$.
 The parameter space for $(\theta,\eta)$ is decomposed as
 \begin{align*}
  \mathbb{R}^p\times \mathbb{R}_+=\mathcal{R}\cup\mathcal{R}^C\text{ and }
  \mathcal{R}\cap\mathcal{R}^C=\emptyset.
 \end{align*}
 Then
\begin{equation}\label{eq:lem:f_around_origin.part.1.1}
 \begin{split}
  & \iint   \eta^{(2p+n)/2} f(\eta\{\|z-\theta\|^2+1\})
  \frac{\bar{\pi}(\eta\|\theta\|^2)}{k(\eta\|\theta\|^2)} \rd \theta   \rd \eta \\
 &= \left(\iint_{\mathcal{R}}+\iint_{\mathcal{R}^C} \right)   \eta^{(2p+n)/2} f(\eta\{\|z-\theta\|^2+1\}) \frac{\bar{\pi}(\eta\|\theta\|^2)}{k(\eta\|\theta\|^2)} \rd \theta   \rd \eta \\
 & \leq \mathcal{J}(f,\pi)\iint_{\mathcal{R}}  \eta^{(2p+n)/2} f(\eta\{\|z-\theta\|^2+1\}) \bar{\pi}(\eta\|\theta\|^2) \rd \theta   \rd \eta \\
 &\qquad +\iint_{\mathcal{R}^C}  \eta^{(2p+n)/2} f(\eta\{\|z-\theta\|^2+1\}) \bar{\pi}(\eta\|\theta\|^2) \rd \theta   \rd \eta \\
&\leq  \left\{\mathcal{J}(f,\pi)+1\right\}\iint   \eta^{(2p+n)/2} f(\eta\{\|z-\theta\|^2+1\}) \bar{\pi}(\eta\|\theta\|^2) \rd \theta   \rd \eta,
 \end{split}
\end{equation}
 which completes the proof of Part \ref{lem:f_around_origin.part.1}.

 For Part \ref{lem:f_around_origin.part.2}, note the following relationship;
 \begin{align*}
  &\iint_\mathcal{R}   \eta^{(2p+n)/2} f(\eta\{\|z-\theta\|^2+1\})
  \frac{\bar{\pi}_i(\eta\|\theta\|^2)}{k(\eta\|\theta\|^2)} \rd \theta   \rd \eta \\
  &\leq \iint_\mathcal{R}   \eta^{(2p+n)/2} f(\eta\{\|z-\theta\|^2+1\})
  \frac{\bar{\pi}_i(\eta\|\theta\|^2)}{k(\eta\|\theta\|^2)} \rd \theta   \rd \eta \\
&\leq \mathcal{J}(f,\pi)\iint_\mathcal{R}   \eta^{(2p+n)/2} f(\eta\{\|z-\theta\|^2+1\}) \bar{\pi}(\eta\|\theta\|^2) \rd \theta   \rd \eta \\
 &= \frac{\mathcal{J}(f,\pi)}{h_1^2(1)}\iint_\mathcal{R}   \eta^{(2p+n)/2} f(\eta\{\|z-\theta\|^2+1\}) \bar{\pi}(\eta\|\theta\|^2)
 h_1^2(1)\rd \theta   \rd \eta \\
 &\leq \frac{\mathcal{J}(f,\pi)}{h_1^2(1)}\iint_\mathcal{R}   \eta^{(2p+n)/2} f(\eta\{\|z-\theta\|^2+1\}) \bar{\pi}(\eta\|\theta\|^2)
 h_i^2(\eta\|\theta\|^2)\rd \theta   \rd \eta \\
 &\leq 64\mathcal{J}(f,\pi)\iint_\mathcal{R}   \eta^{(2p+n)/2} f(\eta\{\|z-\theta\|^2+1\}) \bar{\pi}(\eta\|\theta\|^2)
 h_i^2(\eta\|\theta\|^2)\rd \theta   \rd \eta.
 \end{align*}
 where the first inequality follows from the fact $h_i^2\leq 1$, the second inequality follows from
 Lemma \ref{lem:f_around_origin.0}, the third inequality follows from Part \ref{lem:h_i.1}
 of Lemma \ref{lem:h_i}.
The last inequality follows from Part \ref{lem:h_i.4.5} of Lemma \ref{lem:h_i}.
 Then, as in \eqref{eq:lem:f_around_origin.part.1.1},
 the inequality \eqref{eq:lem:f_around_origin.part.2} can be established.
\end{proof}

\begin{lemma}\label{lem:general_integration_by_parts}
Under Assumptions \ref{FF1}--\ref{FF3} on $f$ and Assumptions \ref{AA1}, \ref{AA2}  \ref{AA3} on $\pi$,  
 \begin{align*}
& z^\T\iint 
  \eta^{(2p+n)/2-1} F(\eta\{\|z-\theta\|^2+1\}) 
\nabla_\theta \bar{\pi}(\eta\|\theta\|^2) \rd \theta   \rd \eta \\
& =\iint 
 \eta^{(2p+n)/2} f(\eta\{\|z-\theta\|^2+1\}) \bar{\pi}(\eta\|\theta\|^2) \rd \theta   \rd \eta \\
&\qquad -(p+n)\iint 
\eta^{(2p+n)/2-1} F(\eta\{\|z-\theta\|^2+1\}) \bar{\pi}(\eta\|\theta\|^2) \rd \theta   \rd \eta .
 \end{align*}
\end{lemma}

 \begin{proof}
First, note the following relationship;
 \begin{align*}
& z^\T\iint 
  \eta^{(2p+n)/2-1} F(\eta\{\|z-\theta\|^2+1\}) 
\nabla_\theta \bar{\pi}(\eta\|\theta\|^2) \rd \theta   \rd \eta \\
& =z^\T\iint 
  \eta^{(2p+n)/2-1} F(\eta\{\|z-\theta\|^2+1\}) 
2\theta\eta \bar{\pi}'(\eta\|\theta\|^2) \rd \theta   \rd \eta \\
& =2z^\T\iint 
  \eta^{(p+n)/2} F(\|\eta^{1/2}z-\mu\|^2+\eta) 
\eta^{-1/2}\mu \bar{\pi}'(\|\mu\|^2) \rd \mu   \rd \eta \\
& =2\iint 
  \eta^{(p+n)/2} F(\|\mu\|^2+\eta) 
\eta^{-1/2}z^\T(\mu+\eta^{1/2}z) \bar{\pi}'(\|\mu+\eta^{1/2}z\|^2) \rd \mu   \rd \eta \\
& =2\iint 
  \eta^{(p+n)/2} F(\|\mu\|^2+\eta) 
  \frac{\partial}{\partial \eta}\bar{\pi}(\|\mu+\eta^{1/2}z\|^2) \rd \mu   \rd \eta.
 \end{align*}
By an integration by parts, the integral with respect to $\eta$ in the above is   
  \begin{align*}
&\int_0^\infty   \eta^{(p+n)/2} F(\|\mu\|^2+\eta) 
  \frac{\partial}{\partial \eta}\bar{\pi}(\|\mu+\eta^{1/2}z\|^2) \rd \eta \\
&=\left[\eta^{(p+n)/2} F(\|\mu\|^2+\eta)   \bar{\pi}(\|\mu+\eta^{1/2}z\|^2) \right]_0^\infty  \\
  &\quad +\frac{1}{2}\int_0^\infty
 \eta^{(p+n)/2} f(\|\mu\|^2+\eta) \bar{\pi}(\|\mu+\eta^{1/2}z\|^2)  \rd \eta \\
& \quad -\frac{p+n}{2}\int_0^\infty \eta^{(p+n)/2-1} F(\|\mu\|^2+\eta) \bar{\pi}(\|\mu+\eta^{1/2}z\|^2) \rd \mu   \rd \eta,
  \end{align*}
  where the first term becomes zero for any fixed $\mu$ under Assumptions.
Then
  \begin{align*}
& z^\T\iint 
  \eta^{(2p+n)/2-1} F(\eta\{\|z-\theta\|^2+1\}) 
\nabla_\theta \bar{\pi}(\eta\|\theta\|^2) \rd \theta   \rd \eta \\
&=\iint
 \eta^{(p+n)/2} f(\|\mu\|^2+\eta) \bar{\pi}(\|\mu+\eta^{1/2}z\|^2)  \rd \eta \\
& \qquad -(p+n)\iint \eta^{(p+n)/2-1} F(\|\mu\|^2+\eta) \bar{\pi}(\|\mu+\eta^{1/2}z\|^2) \rd \mu   \rd \eta \\
   &=\iint 
 \eta^{(2p+n)/2} f(\eta\{\|z-\theta\|^2+1\}) \bar{\pi}(\eta\|\theta\|^2) \rd \theta   \rd \eta \\
&\qquad -(p+n)\iint 
\eta^{(2p+n)/2-1} F(\eta\{\|z-\theta\|^2+1\}) \bar{\pi}(\eta\|\theta\|^2) \rd \theta   \rd \eta, 
  \end{align*}
  which completes the proof.
 \end{proof}

\section{Proof of Corollary \ref{thm:interesting.2.main}, Part \ref{thm:interesting.2.3}}
\label{sec.ap.alam}
Let $ \phi_\alpha(w)=w\psi_\alpha(w)$ where, as in \eqref{psialphageneral}, 
\begin{equation*}
\psi_\alpha(w)=
  \frac{\int_0^1 t^{p/2-\alpha-1}(1-t)^{\alpha}(1+wt)^{-(p+n)/2-1}\rd t}
  {\int_0^1 t^{p/2-\alpha-2}(1-t)^{\alpha}(1+wt)^{-(p+n)/2-1}\rd t},
\end{equation*}
for $\alpha\in(-1,0)$. \cite{Maruyama-Strawderman-2009} showed that
\begin{enumerate}
 \item $\phi_\alpha(w)$ is not monotonic, 
 \item $ 0\leq \phi_\alpha(w)\leq \phi_\star(\alpha)$ 
\begin{equation*}
 \phi_\star(\alpha)=\frac{p/2-\alpha-1}{n/2+\alpha+1+\alpha(p/2+n/2)},
\end{equation*}
 \item $w\phi'_\alpha(w)/\phi_\alpha(w)\geq -c(\alpha)$ 
\begin{equation*}
 c(\alpha)=-\frac{(p/2-\alpha)\alpha}{\alpha+1}.
\end{equation*}
\end{enumerate}
in Part (iii) of Corollary 3.1, Part (iv) of Corollary 3.1 and Lemma 3.4, respectively.
\cite{Kubokawa-2009} proposed a sufficient condition of $\delta_\phi=\{1-\phi(W)/W\}X$
to be minimax as follows;
\begin{equation*}
 w\phi'(w)/\phi(w)\geq -c,\text{ and }0\leq \phi\leq 2\frac{p-2-2c}{n+2+2c},
  \text{ for some }c>0,
\end{equation*}
where the upper bound is larger than the upper bound which \cite{Maruyama-Strawderman-2009}
and \cite{Wells-Zhou-2008} applied.
Note that $\phi_\star(\alpha)$ is increasing in $\alpha\in(-1/2,0)$
and that $c(\alpha)$ is decreasing in $\alpha\in(-1/2,0)$.
Then the inequalities
\begin{equation*}
\begin{split}
 \frac{p/2-\alpha-1}{n/2+\alpha+1+\alpha(p/2+n/2)}&\leq 2\frac{p-2-2c(\alpha)}{n+2+2c(\alpha)}\\
&  =2\frac{(p-2)(\alpha+1)+2(p/2-\alpha)\alpha}{(n+2)(\alpha+1)-2(p/2-\alpha)\alpha}
\end{split} 
\end{equation*}
as well as $ -1<\alpha<0$ are a sufficient condition for minimaxity of $\delta_{\psi_\alpha}$.
Let
\begin{align*}
 f(\alpha)&=
2\left\{(p-2)(\alpha+1)+2(p/2-\alpha)\alpha\right\}\left\{n/2+\alpha+1+\alpha(p/2+n/2)\right\} \\
 & \quad - \left\{(n+2)(\alpha+1)-2(p/2-\alpha)\alpha\right\}\left(p/2-\alpha-1\right) \\
 &=\frac{(p-2)(n+2)}{2}+\frac{(n+2)(5p-8)+3p(p-2)}{2}\alpha \\
 &\quad + \left\{2(p-1)^2+(2p-3)(n+2)\right\}\alpha^2-2(p+n+1)\alpha^3.
\end{align*}
For $\alpha\in(-1/2,0)$,
\begin{equation*}
 f(\alpha)\geq \frac{(p-2)(n+2)}{2}+\frac{5(n+2)(p-2)+2(n+2)+3p(p-2)}{2}\alpha 
\end{equation*}
which is nonnegative when
\begin{equation*}
-\left(5+\frac{2}{p-2}+\frac{3p}{n+2}\right)^{-1} \leq \alpha <0.
\end{equation*}
Hence Part \ref{thm:interesting.2.3} follows.

 \section{Proof of Corollary \ref{thm:interesting.3.main}}
\label{sec.ap.2005}
Let 
\begin{equation*}
 \bar{\pi}(\eta\|\theta\|^2)=\int_b^\infty \frac{1}{(2\pi)^{p/2}\xi^{p/2}}
  \exp\left(-\frac{\eta\|\theta\|^2}{2\xi}\right)g_\pi(\xi)\rd \xi
\end{equation*}
where
\begin{align*}
 g_\pi(\xi)=(\xi- b)^\alpha(\xi+1)^\beta.
\end{align*}
Eventually set $ -1<\alpha\leq n/2$, $\beta=-n/2$ and $b\geq 0$.

Note the underlying density is Gaussian and let
\begin{align*}
 f_G(t)=\frac{1}{(2\pi)^{(p+n)/2}}\exp(-t/2).
\end{align*}
Note
\begin{align*}
 \|z-\theta\|^2+\frac{\|\theta\|^2}{\xi}=\frac{\xi+1}{\xi}\left\|\theta -\frac{\xi}{\xi+1}z\right\|^2
 +\frac{\|z\|^2}{\xi+1}.
\end{align*}
Then we have
\begin{align*}
 & M_1(z;\pi)\\
 &=\frac{1}{(2\pi)^{(p+n)/2}}
 \int_b^\infty\int_0^\infty\eta^{(p+n)/2}\exp\left(-\frac{\eta}{2}\right)g_\pi(\xi) \\ &\quad\times  \left\{\int_{\mathbb{R}^p}
 \frac{\eta^{p/2}}{(2\pi)^{p/2}\xi^{p/2}}
 \exp\left(-\eta\frac{\|z-\theta\|^2}{2}-\frac{\eta\|\theta\|^2}{2\xi}\right)\rd\theta\right\}\rd \eta \rd\xi\\
 &=\frac{1}{(2\pi)^{(p+n)/2}}
 \int_b^\infty\int_0^\infty\eta^{(p+n)/2}\exp\left(-\frac{\eta}{2}\right)g_\pi(\xi) \\ &\quad\times
\frac{1}{(\xi+1)^{p/2}}\exp\left(-\frac{\eta\|z\|^2}{2(\xi+1)}\right)
\rd \eta\rd\xi \\
 &=c
 \int_b^\infty
\left(1+\frac{\|z\|^2}{\xi+1}\right)^{-(p+n)/2-1}
\frac{g_\pi(\xi) }{(\xi+1)^{p/2}}\rd\xi \\
 &=c
 \int_b^\infty\frac{(\xi-b)^\alpha(1+\xi)^{\beta+n/2+1}}{(1+\xi+\|z\|^2)^{(p+n)/2+1}}
\rd\xi \\
 &=c
 \int_0^\infty\frac{\xi^\alpha(1+b+\xi)^{\beta+n/2+1}}{(1+b+\|z\|^2+\xi)^{(p+n)/2+1}}
\rd\xi 
\end{align*}
where
\begin{align*}
 c=\frac{\Gamma((p+n)/2+1)2^{(p+n)/2+1}}{(2\pi)^{(p+n)/2}}.
\end{align*}
Also we have
\begin{align*}
 & z^\T M_2(z;\pi)\\
 &=\frac{1}{(2\pi)^{(p+n)/2}}
 \int_b^\infty\int_0^\infty\eta^{(p+n)/2}\exp\left(-\frac{\eta}{2}\right)g_\pi(\xi) \\ &\quad\times  \left\{\int_{\mathbb{R}^p}
 \frac{z^\T\theta\eta^{p/2}}{(2\pi)^{p/2}\xi^{p/2}}
 \exp\left(-\eta\frac{\|z-\theta\|^2}{2}-\frac{\eta\|\theta\|^2}{2\xi}\right)\rd\theta\right\}\rd \eta \rd\xi\\
 &=\frac{\|z\|^2}{(2\pi)^{(p+n)/2}}
 \int_b^\infty\int_0^\infty\eta^{(p+n)/2}\exp\left(-\frac{\eta}{2}\right)g_\pi(\xi) \\ &\quad\times
\frac{1}{(\xi+1)^{p/2}}\frac{\xi}{\xi+1}\exp\left(-\frac{\eta\|z\|^2}{2(\xi+1)}\right)
 \rd \eta\rd\xi \\
 &=\|z\|^2M_1(z;\pi)-\|z\|^2
c \int_0^\infty\frac{\xi^\alpha(1+b+\xi)^{\beta+n/2}}{(1+b+\|z\|^2+\xi)^{(p+n)/2+1}}
\rd\xi .
\end{align*}
Recall
\begin{align*}
  \psi_\pi(z)=\frac{z^\T z M_1(z,\pi)- z^\T M_2(z,\pi)}{\|z\|^2 M_1(z,\pi)}.
\end{align*}
Under the choice $\beta=-n/2$, we have
\begin{align*}
 \psi_\pi(z)&=
 \frac{\int_0^\infty\xi^\alpha(1+ b +\|z\|^2+\xi)^{-(p+n)/2-1}\rd\xi}
 {\int_0^\infty\xi^\alpha(1+ b +\xi)(1+ b +\|z\|^2+\xi)^{-(p+n)/2-1}\rd\xi} \\
 &=\left(1+c+\frac{\int_0^\infty\xi^{\alpha+1}(1+ b +\|z\|^2+\xi)^{-(p+n)/2-1}\rd\xi}
 {\int_0^\infty\xi^\alpha(1+ b +\|z\|^2+\xi)^{-(p+n)/2-1}\rd\xi}\right)^{-1} \\
 &=\left(1+ b +(1+ b +\|z\|^2)\frac{B(\alpha+2,(p+n)/2-\alpha-1)}{B(\alpha+1,(p+n)/2-\alpha)}
\right)^{-1} \\
 &=\left(1+ b +(1+ b +\|z\|^2)\frac{\alpha+1}{(p+n)/2-\alpha-1}
 \right)^{-1}.
\end{align*}
Let $a=\{(p+n)/2-\alpha-1\}/(\alpha+1)$.
Then the Bayes equivariant estimator is
\begin{align}\label{simple.Bayes.estimator}
 \left(1-\frac{a}{\|X\|^2/S+(a+1) (b+1)}\right)X.
\end{align}
When $\alpha+\beta<-1$ or equivalently $ \alpha<n/2-1$ as well as $\alpha>-1$,
this is a proper Bayes equivariant estimator.
When $-1\leq \alpha+\beta\leq 0$ or equivalently
$n/2-1\leq \alpha\leq n/2$ as well as $\alpha>-1/2$,
this is an admissible generalized Bayes equivariant estimator.
Hence when $a\geq (p-2)/(n+2)$ and $b\geq 0$,
the estimator \eqref{simple.Bayes.estimator} is admissible within the class of
equivariant estimators.
 
 \section{Proof of satisfaction of \ref{AA1}--\ref{AA3} by \eqref{billsprior}}
 \label{bill-1971}
\begin{lemma}
 Let $\pi(\lambda)=c_p\lambda^{p/2-1}\bar{\pi}(\lambda)$ where
 \begin{equation}\label{barpibilllem}
  \bar{\pi}(\lambda)=\int_b^\infty\frac{1}{(2\pi\xi)^{p/2}}\exp\left(-\frac{\lambda}{2\xi}\right)
  (\xi-b)^\alpha(1+\xi)^\beta\rd \xi.
 \end{equation}  
 Then Assumptions \ref{AA1}--\ref{AA3} are satisfied when
 $\{b>0,\ -1\leq \alpha+\beta\leq 0,\text{ and }\alpha>-1\} $ or
$\{b=0,\ -1\leq \alpha+\beta\leq 0,\text{ and }\alpha>-1/2\} $.
\end{lemma}
Note the integrability of \eqref{barpibilllem} under $(b,b+\epsilon)$ follows
$b=0$ or $b>0$ as well as $\alpha>-1$.
Also the integrability of \eqref{barpibilllem} under $(b+\epsilon,\infty)$ follows
when
$\alpha+\beta-p/2<-1$.
Further note Note also if $\alpha+\beta< -1$ as well as $\alpha>-1$, 
$\int_b^\infty(\xi-b)^\alpha(1+\xi)^\beta\rd \xi<\infty$ and hence the prior on $\lambda$ is proper.

\begin{proof}
 Clearly $\pi(\lambda)$ is differentiable as
\begin{equation}\label{eq:barpiderivative}
 \bar{\pi}'(\lambda)=-\frac{1}{2}
\int_b^\infty\frac{(2\pi)^{-p/2}}{\xi^{p/2+1}}\exp\left(-\frac{\lambda}{2\xi}\right)
  (\xi-b)^\alpha(1+\xi)^\beta\rd \xi.
\end{equation}  
 and hence Assumption \ref{AA1} is satisfied.

 [Assumption \ref{AA2} with $\alpha>-1/2$ and $b=0$]
When $b=0$, by Tauberian theorem, we have, in \eqref{barpibilllem} and \eqref{eq:barpiderivative},
\begin{equation}\label{eq:tau.0}
 \begin{split}
\lim_{\lambda\to 0}\frac{(2\pi)^{p/2}\bar{\pi}(\lambda)}{(\lambda/2)^{-(p/2-\alpha-1)}\Gamma(p/2-\alpha-1)}=1 \\
\lim_{\lambda\to 0}\frac{-2(2\pi)^{p/2}\bar{\pi}'(\lambda)}{(\lambda/2)^{-(p/2-\alpha)}\Gamma(p/2-\alpha)}=1,
\end{split}
\end{equation}
which implies that
\begin{equation}\label{eq:tau.1}
 \lim_{\lambda\to 0}\lambda\frac{\bar{\pi}'(\lambda)}{\bar{\pi}(\lambda)}=-\frac{p}{2}+\alpha+1.
\end{equation}
Recall $ \pi(\lambda)=c_p\lambda^{p/2-1}\bar{\pi}(\lambda)$ and let $\nu(\lambda)=\lambda^{-\alpha}\pi(\lambda)=c_p\lambda^{p/2-1-\alpha}\bar{\pi}(\lambda)$. Then we have
\begin{align*}
 \nu(0)=c_p2^{p/2-\alpha-1}\Gamma(p/2-\alpha-1)(2\pi)^{-p/2}
\end{align*}
by \eqref{eq:tau.0} and 
\begin{align*}
 \lim_{\lambda\to 0}\lambda\frac{\nu'(\lambda)}{\nu(\lambda)}=
 \lim_{\lambda\to 0}\lambda\nu'(\lambda)=0
\end{align*}
by \eqref{eq:tau.1}.

 [Assumption \ref{AA2} with $\alpha>-1$ and $b>0$]
 When $\alpha>-1$ and $b>0$, it follows that $0<\bar{\pi}(0)<\infty$ and $0<|\bar{\pi}'(0)|<\infty$.
 For this case take $ \nu(\lambda)=c_p\bar{\pi}(\lambda)$.
 Then $\pi(\lambda)=\lambda^{p/2-1}\nu(\lambda)$, where $p/2-1>-1/2$ and
$\nu(\lambda)$ satisfies 
\begin{equation*}
 0<\nu(0)<\infty \text{ and }
   \lim_{\lambda\to 0}\lambda\frac{\nu'(\lambda)}{\nu(\lambda)}=0.
\end{equation*}

[Assumption \ref{AA3}]
By Tauberian theorem, we have, in \eqref{barpibilllem} and \eqref{eq:barpiderivative},
\begin{equation}\label{eq:tau.2}
 \begin{split}
\lim_{\lambda\to \infty}\frac{(2\pi)^{p/2}\bar{\pi}(\lambda)}{(\lambda/2)^{-(p/2-\alpha-\beta-1)}\Gamma(p/2-\alpha-\beta-1)}=1 \\
\lim_{\lambda\to \infty}\frac{-2(2\pi)^{p/2}\bar{\pi}'(\lambda)}{(\lambda/2)^{-(p/2-\alpha-\beta)}\Gamma(p/2-\alpha-\beta)}=1,
\end{split}
\end{equation}
which implies that
\begin{equation}\label{eq:tau.3}
 \lim_{\lambda\to \infty}\lambda\frac{\bar{\pi}'(\lambda)}{\bar{\pi}(\lambda)}=-\frac{p}{2}+\alpha+\beta+1\text{ and }
\lim_{\lambda\to \infty}\lambda\frac{\pi'(\lambda)}{\pi(\lambda)}=\alpha+\beta.
\end{equation}
Hence when $-1\leq \alpha+\beta<0$, Assumption \ref{AA3.1} is satisfied.

When $\alpha+\beta=0$, note 
\begin{equation*}
 (\xi-b)^\alpha(1+\xi)^\beta=
  \left(1-\frac{1+b}{1+\xi}\right)^\alpha
\end{equation*}
and
\begin{equation*}
 \lim_{\xi\to\infty}\xi\left\{\left(1-\frac{1+b}{1+\xi}\right)^\alpha-1\right\}=-\alpha(1+b).
\end{equation*}
Then
\begin{equation*}
 \lim_{\xi\to\infty}
  \frac{(2\pi)^{p/2}\bar{\pi}(\lambda)-(\lambda/2)^{-(p/2-1)}\Gamma(p/2-1)}
  {(\lambda/2)^{-p/2}\Gamma(p/2)}=-\alpha(1+b)
\end{equation*}
and
\begin{equation*}
 \lim_{\xi\to\infty}
  \frac{-2(2\pi)^{p/2}\bar{\pi}'(\lambda)-(\lambda/2)^{-p/2}\Gamma(p/2)}
  {(\lambda/2)^{-p/2-1}\Gamma(p/2+1)}=-\alpha(1+b).
\end{equation*}
Hence we get
\begin{equation*}
 \lim_{\xi\to\infty}\lambda\left(\lambda\frac{\bar{\pi}'(\lambda)}{\bar{\pi}(\lambda)}+\frac{p}{2}-1\right)
=\lim_{\xi\to\infty}\lambda^2\frac{\pi'(\lambda)}{\pi(\lambda)}  =2(p-2)\alpha(b+1),
\end{equation*}
which satisfies Assumption \ref{AA3.2.2}.
Hence Assumption \ref{AA3} is satisfied by $-1\leq \alpha+\beta\leq 0$.
\end{proof}

 \section{Proof of Admissibility of $X$ for $p=1,2$}
 \label{p12}
 In the Gaussian case, $X\sim N_p(\theta,\eta^{-1}I)$ and $\eta \|U\|^2\sim \chi^2_n$,
 Kubokawa in his unpublished lecture note written in Japanese,
 showed that when $p=1,2$, the estimator $X$ is admissible among all estimators.
Here we generalize it for our general situation with the underlying density $f$ given by \eqref{eq:density.f}.
For a general prior $g(\theta,\eta)$, 
we have
\begin{equation*}
\begin{split}
 \delta_g (x,u) &=\frac
 {\int_{\mathbb{R}^p}\int_0^\infty \theta \eta \eta^{(p+n)/2}f(\eta\{\|x-\theta\|^2+\|u\|^2\})g(\theta,\eta)\rd \theta \rd \eta}
{\int_{\mathbb{R}^p}\int_0^\infty \eta \eta^{(p+n)/2}f(\eta\{\|x-\theta\|^2+\|u\|^2\})g(\theta,\eta)\rd \theta \rd \eta} \\
&=x + \frac{\int_{\mathbb{R}^p}\int_0^\infty (\theta-x) \eta \eta^{(p+n)/2}f(\eta\{\|x-\theta\|^2+\|u\|^2\})g(\theta,\eta)\rd \theta \rd \eta}{\int_{\mathbb{R}^p}\int_0^\infty \eta^{(p+n)/2+1} f(\eta\{\|x-\theta\|^2+\|u\|^2\})g(\theta,\eta)\rd \theta \rd \eta} \\
 &=
x - \frac{\int_{\mathbb{R}^p}\int_0^\infty \eta^{(p+n)/2}\nabla_\theta F(\eta\{\|x-\theta\|^2+\|u\|^2\})g(\theta,\eta)\rd \theta \rd \eta}{\int_{\mathbb{R}^p}\int_0^\infty \eta^{(p+n)/2+1} f(\eta\{\|x-\theta\|^2+\|u\|^2\})g(\theta,\eta)\rd \theta \rd \eta} \\
 &=
x + \frac{\int_{\mathbb{R}^p}\int_0^\infty  \eta^{(p+n)/2}F(\eta\{\|x-\theta\|^2+\|u\|^2\})\nabla_\theta g(\theta,\eta)\rd \theta \rd \eta}{\int_{\mathbb{R}^p}\int_0^\infty \eta^{(p+n)/2+1} f(\eta\{\|x-\theta\|^2+\|u\|^2\})g(\theta,\eta)\rd \theta \rd \eta}, 
\end{split}
\end{equation*}
where $F(t)=(1/2)\int_t^\infty f(s)\rd s$ and the last equality follows from an integration by parts. 
Hence 
the estimator $X$ is the generalized Bayes estimator
with respect to any improper prior which does not depend on $\theta$, say $g(\theta,\eta)=\pi(\eta)$.
Further let
\begin{align*}
 g_i(\theta,\eta)=h^2_i(\eta\|\theta\|^2)\pi(\eta)
\end{align*}
where $h_i$ is given by \eqref{eq:new_hir}.
Clearly $g_i(\theta,\eta)$ approaches $\pi(\eta)$ as $i\to\infty$.
Also $ g_i(\theta,\eta)$ for any fixed $i$ is integrable under the condition
\begin{equation}\label{eq:pieta}
 \int_0^\infty \eta^{-p/2}\pi(\eta)\rd \eta<\infty
\end{equation}
since
\begin{align*}
\int_{\mathbb{R}^p}\int_0^\infty g_i(\theta,\eta)\rd\theta\rd\eta 
&=\int_{\mathbb{R}^p}\int_0^\infty \eta^{p/2}h^2_i(\eta\|\theta\|^2)\eta^{-p/2}\pi(\eta)\rd\theta\rd\eta \\
&=c_p\int_0^\infty \lambda^{p/2-1} h^2_i(\lambda)\rd\lambda\int_0^\infty\eta^{-p/2}\pi(\eta)\rd\eta,
\end{align*}
where, by Lemma \ref{lem:h_i} of  Appendix \ref{sec:assumption},
$\int_0^\infty \lambda^{p/2-1} h^2_i(\lambda)\rd\lambda<\infty $ for $p=1,2$. 
\begin{thm}
Assume Assumptions \ref{FF1}, \ref{FF2} and \ref{FF3.1} on $f$.
Then the estimator $X$ is admissible for $p=1,2$. 
\end{thm}
\begin{proof}
Let $ \delta_{gi}$ be the proper Bayes estimator with respect to $g_i(\theta,\eta)$.
Then the Bayes risk difference
of $ x$ and $ \delta_{gi}$ with respect to $ g_i(\theta,\eta)$  is 
\begin{equation*}
\begin{split}
 \Delta_i &= \int_{\mathbb{R}^{p}}\int_0^\infty \left\{
R(\theta,\eta,X)-R(\theta,\eta,\delta_{gi}) \right\}
g_i(\theta,\eta)
 \rd \theta \rd\eta\\
&= \int_{\mathbb{R}^{p}}\int_{\mathbb{R}^{n}}\int_{\mathbb{R}^{p}}\int_0^\infty 
\eta\left\{\|x-\theta\|^2-\|\delta_{gi}-\theta\|^2\right\}\\
&\quad\times\eta^{(p+n)/2} f(\eta\{\|x-\theta\|^2+\|u\|^2\})g_i(\theta,\eta)
\rd x\rd u \rd \theta \rd\eta\\
 &= \int_{\mathbb{R}^{p}}\int_{\mathbb{R}^{n}}\int_{\mathbb{R}^{p}}\int_0^\infty \| \delta_{gi}-x \|^2 \\
&\quad\times\eta^{(p+n)/2+1} f(\eta\{\|x-\theta\|^2+\|u\|^2\})g_i(\theta,\eta)
\rd x\rd u \rd \theta \rd\eta\\
 &= \int_{\mathbb{R}^{p}}\int_{\mathbb{R}^{n}}
 \frac{\left\|\int_{\mathbb{R}^p}\int_0^\infty  \eta^{(p+n)/2}F(\eta\{\|x-\theta\|^2+\|u\|^2\})\nabla_\theta g_i(\theta,\eta)\rd\theta \rd\eta\right\|^2}{\int_{\mathbb{R}^p}\int_0^\infty \eta^{(p+n)/2+1} f(\eta\{\|x-\theta\|^2+\|u\|^2\})g_i(\theta,\eta)\rd\theta \rd\eta}
 \rd x\rd u.
\end{split}
\end{equation*}
 Note
 \begin{align*}
  \nabla_\theta g_i(\theta,\eta)=4\eta\theta h_i(\eta\|\theta\|^2)h'_i(\eta\|\theta\|^2).
 \end{align*}
Then, by Cauchy-Schwarz inequality, we have
\begin{equation}
\begin{split}
 \Delta_i &\leq \int_{\mathbb{R}^{p}}\int_{\mathbb{R}^{n}}\int_{\mathbb{R}^{p}}\int_0^\infty
 \eta^{(p+n)/2-1}\frac{F(\eta\{\|x-\theta\|^2+\|u\|^2\})^2}{f(\eta\{\|x-\theta\|^2+\|u\|^2\})}\\
&\quad\times \frac{\|\nabla_\theta g_i(\theta,\eta)\|^2}{g_i(\theta,\eta)}
  \rd x\rd u \rd\theta \rd\eta \\
 &=c_{p+n}\int_0^\infty t^{(p+n)/2-1}\frac{F^2(t)}{f(t)}\rd t
 \int_{\mathbb{R}^{p}}\int_0^\infty
\frac{\|\nabla_\theta g_i(\theta,\eta)\|^2}{\eta g_i(\theta,\eta)}
 \rd\theta \rd\eta \\
 &=16 c_{p+n}A_2 \int_{\mathbb{R}^{p}}\int_0^\infty \eta^{p/2}\eta\|\theta\|^2 \left\{h'_i(\eta\|\theta\|^2)\right\}^2
\eta^{-p/2}\pi(\eta)\rd\theta \rd\eta \\
 &=16 c_p c_{p+n}A_2 \int_0^\infty \frac{\pi(\eta)}{\eta^{p/2}}\rd \eta
 \int_0^\infty\lambda^{p/2} \sup_i\left\{h'_i(\lambda)\right\}^2\rd \lambda
\end{split}
\end{equation}
 where $A_2=\int_0^\infty t^{(p+n)/2-1}\{F^2(t)/f(t)\}\rd t$ and it is bounded under
 Assumptions \ref{FF1}, \ref{FF2} and \ref{FF3.1} on $f$,
 as in Part \ref{lem:f.1.2} of Lemma \ref{lem:f}.
 Further, 
by Lemma \ref{lem:h_i} of  Appendix \ref{sec:assumption},
$\int_0^\infty\lambda^{p/2} \sup_i\left\{h'_i(\lambda)\right\}^2\rd \lambda<\infty $ for $p=1,2$. 
 Hence, by the dominated convergence theorem,
 we have $\Delta_i\to 0$ as $i\to\infty$. By the Blyth sufficient condition,
 the admissibility of $X$ for $p=1,2$ follows.
 \end{proof}

\end{document}